\documentclass[12pt,a4paper]{amsart}
\usepackage{amssymb}
\usepackage{mathrsfs}
\usepackage{xcolor}
\usepackage{enumerate}
\usepackage{slashed}
\usepackage[margin=2.5cm]{geometry}

\allowdisplaybreaks

\numberwithin{equation}{section}
\newtheorem{theorem}{Theorem}[section]
\newtheorem{lemma}[theorem]{Lemma}
\newtheorem{corollary}[theorem]{Corollary}
\newtheorem{proposition}[theorem]{Proposition}

\theoremstyle{definition}

\theoremstyle{remark}

\newcommand{\sfrac}[2]{{#1}/{#2}}

\newcommand{\lnorm}{\left\Vert}
\newcommand{\rnorm}{\right\Vert}

\newcommand{\Biglnorm}{\Bigl\Vert}
\newcommand{\Bigrnorm}{\Bigr\Vert}

\newcommand{\abs}[1]{\left| #1 \right|}
\newcommand{\labs}{\left|}
\newcommand{\rabs}{\right|}
\newcommand{\biglabs}{\bigl|}
\newcommand{\bigrabs}{\bigr|}
\newcommand{\Biglabs}{\Bigl|}
\newcommand{\Bigrabs}{\Bigr|}
\newcommand{\bigglabs}{\biggl|}
\newcommand{\biggrabs}{\biggr|}

\newcommand{\lset}{\left\{}
\newcommand{\rset}{\right\}}
\newcommand{\biglset}{\bigl\{}
\newcommand{\bigrset}{\bigr\}}
\newcommand{\Biglset}{\Bigl\{}
\newcommand{\Bigrset}{\Bigr\}}
\newcommand{\bigglset}{\biggl\{}
\newcommand{\biggrset}{\biggr\}}

\newcommand{\lpar}{\left( }
\newcommand{\rpar}{\right) }
\newcommand{\biglpar}{\bigl(}
\newcommand{\bigrpar}{\bigr)}
\newcommand{\Biglpar}{\Bigl( }
\newcommand{\Bigrpar}{\Bigr) }
\newcommand{\bigglpar}{\biggl(}
\newcommand{\biggrpar}{\biggr)}

\newcommand{\norm}[1]{\left\Vert #1 \right\Vert}

\newcommand{\supp}{\operatorname{supp}}

\newcommand{\term}[1]{\mathrm{#1}}
\newcommand{\Lie}[1]{\mathfrak{#1}}

\newcommand{\Tau}{\mathrm{T}}
\newcommand{\R}{\mathbb{R}}
\newcommand{\N}{\mathbb{N}}

\newcommand{\C}{\mathbb{C}}
\newcommand{\Z}{\mathbb{Z}}
\newcommand{\Rplus}{\R_{+}}

\newcommand{\dcubes}{\mathscr{Q}}
\newcommand{\maxrect}{\mathscr{M}}
\newcommand{\drect}{\mathscr{R}}
\newcommand{\rect}{\mathscr{P}}
\newcommand{\bases}{\mathscr{B}}

\newcommand{\len}{\operatorname{\ell}}

\newcommand{\indext}{\tau}
\newcommand{\Indext}{\Tau}
\newcommand{\homodim}{\nu}
\newcommand{\wrt}{\,\mathrm{d}}
\newcommand{\wrtprotonprot}{\,\frac{\mathrm{d}\prot}{\prot}}
\newcommand{\expe}{\mathrm{e}}

\newcommand{\loc}{\mathrm{loc}}

\newcommand{\one}{^{[1]}}
\newcommand{\two}{^{[2]}}
\newcommand{\blen}{\operatorname{\boldsymbol{\ell}}}
\newcommand{\proG}{\mathbf{G}}
\newcommand{\proT}{\mathbf{T}}
\newcommand{\prog}{\mathbf{g}}
\newcommand{\proh}{\mathbf{h}}

\newcommand{\proo}{\mathbf{o}}
\newcommand{\pror}{\mathbf{r}}
\newcommand{\prot}{\mathbf{t}}
\newcommand{\proz}{\mathbf{z}}

\def\Xint#1{\mathchoice
{\XXint\displaystyle\textstyle{#1}}%
{\XXint\textstyle\scriptstyle{#1}}%
{\XXint\scriptstyle\scriptscriptstyle{#1}}%
{\XXint\scriptscriptstyle\scriptscriptstyle{#1}}%
\!\int}
\def\XXint#1#2#3{{\setbox0=\hbox{$#1{#2#3}{\int}$}
\vcenter{\hbox{$#2#3$}}\kern-.5\wd0}}

\def\dashint{\Xint-}

\newcommand{\fullnabla}{\slashed{\nabla}}

\newcommand{\psitt}{\psi_{\prot}}
\newcommand{\phitt}{\phi_{\prot}}

\newcommand{\pois}{p}
\newcommand{\poist}{\pois_{t}}

\newcommand{\poisone}{\pois_1}
\newcommand{\poisoneone}{\pois_1\one}
\newcommand{\poisonetwo}{\pois_1\two}
\newcommand{\ptt}{\pois_{\prot}}
\newcommand{\ptone}{\pois\one_{t_1}}
\newcommand{\pttwo}{\pois\two_{t_2}}
\newcommand{\pti}{\pois^{[i]}_{t_i}}

\newcommand{\conjpois}{q}
\newcommand{\qt}{{{\conjpois}}_{t}}
\newcommand{\qtt}{{{\conjpois}}_{\prot}}
\newcommand{\qone}{{{\conjpois}}_{1}}
\newcommand{\qti}{\conjpois^{[i]}_{t_i}}
\newcommand{\qtone}{\conjpois\one_{t_1}}
\newcommand{\qttwo}{\conjpois\two_{t_2}}

\newcommand{\heat}{h}
\newcommand{\heatt}{\heat_t}
\newcommand{\heatone}{\heat_1}

\newcommand{\fnspace}[1]{\mathsf{#1}}
\newcommand{\BMO}{\fnspace{BMO}}
\newcommand{\Hardy}{\fnspace{H}^1}
\newcommand{\LlogplusL}{\fnspace{L\,log\,L}}
\newcommand{\Leb}[1]{\fnspace{L}^{#1}}

\newcommand{\oper}[1]{\mathcal{#1}}
\newcommand{\Riesz}{\oper{R}}
\newcommand{\Lap}{\oper{L}}
\newcommand{\sqfn}[2]{\oper{S}_{#1,#2}}
\newcommand{\strmaxfn}{\oper{M}_{s}}
\newcommand{\maxfn}[2]{\oper{M}_{#1,#2}}

\let\oldphi\phi \let\oldvarphi\varphi
\renewcommand{\phi}{\oldvarphi}
\renewcommand{\varphi}{\oldphi}

\let\oldepsilon\epsilon \let\oldvarepsilon\varepsilon
\renewcommand{\epsilon}{\oldvarepsilon}
\renewcommand{\varepsilon}{\oldepsilon}

\begin{document}

\title[Endpoint estimate for product singular integrals]
{An endpoint estimate for product singular integral operators
on stratified Lie groups}

\author{Michael G. Cowling}
\address{School of Mathematics and Statistics, University of New South Wales, Sydney 2052, Australia}
\email{m.cowling@unsw.edu.au}

\author{Ming-Yi Lee}
\address{Department of Mathematics, National Central University, Chung-Li 320, Taiwan, Republic of China}
\email{mylee@math.ncu.edu.tw}

\author{Ji Li}
\address{Ji Li, Department of Mathematics, Macquarie University, NSW, 2109, Australia}
\email{ji.li@mq.edu.au}

\author{Jill Pipher}
\address{Department of Mathematics, Brown University, Providence RI 02912, USA}
\email{jill\_pipher@brown.edu}

\thanks{Cowling and Li are supported by ARC DP 220100285. 
Lee is supported by NSTC 112-2115-M-008-001-MY2.}


\begin{abstract}
We establish hyperweak boundedness of area functions, square functions, maximal operators and Calderón--Zygmund operators on products of two stratified Lie groups.
\end{abstract}

\maketitle

\section{Introduction and statement of main results}

Since the 1980s, the development of multiparameter harmonic analysis has proceeded apace; recent contributions in the area include \cite{CF1,F1,F2,fl, HyM, Jo, lppw, NS, Ou, P}.
Much of the product space theory on $\R^m \times \R^n$, including the duality of $\Hardy$ with $\BMO$, characterisation of $\Hardy$ by square functions and atomic decompositions, and interpolation,  has been extended to more general products of spaces of homogeneous type.

In contrast to the classical theory of singular integrals, in multiparameter harmonic analysis product singular integrals are not of weak type $(1,1)$. 
For functions supported on the unit cube, the classical weak type $(1,1)$ estimate was replaced by a $\LlogplusL $ to $\Leb{1,\infty}$ estimate by R. Fefferman \cite{F1}. 
More precisely, let $\oper{T}$ be a product Calderón--Zygmund operator on $\Leb 2(\R^m\times\R^n)$ as in \cite{Jo} or \cite{F1}; then for every $\lambda\in \Rplus$,
\begin{align}\label{RFe}
|\{(x_1,x_2)\in \R^m \times \R^n\colon |\oper{T} f(x_1,x_2)|>\lambda\}|
\leq \frac{C}{\lambda} \|f\|_{\LlogplusL (Q_1\times Q_2)}
\end{align}
for all $f$ supported in the product $Q_1 \times Q_2$ of the unit cubes in $\R^m$ and $\R^n$, where
\begin{align*}
 \|f\|_{\LlogplusL (Q_1\times Q_2)} = \iint_{Q_1\times Q_2} |f(x_1,x_2)| (1+\log^+(|f(x_1,x_2)|)) \wrt x_2 \wrt x_1;
\end{align*}
here $\log^+t :=\log(\max\{1,t\})$.
The proof in \cite{F1} relies on the boundedness of the strong maximal function and the area function from $\LlogplusL$ to $\Leb{1,\infty}$, the local atomic decomposition of functions in $\LlogplusL$ produced using the $\LlogplusL $ to $\Leb{1,\infty}$ boundedness of the area function, and the boundedness of $\oper{T}$ on $\LlogplusL$ atoms.

A natural question arises: is \eqref{RFe} true globally for area functions, square functions, maximal operators and singular integral operators, even in more general product settings?

In this paper, we answer the above question positively on product spaces $\proG := G_1 \times G_2$, where $G_1$ and $G_2$ are stratified Lie groups.
To state our results, we need a little notation.
Details may be found later.

In this introduction, the auxiliary functions $\phi^{[i]}$ on $G_i$ satisfy standard decay and smoothness conditions and have integral $1$.
Likewise, the functions $\psi^{[i]}$ on $G_i$ satisfy standard decay and smoothness and have integral $0$.
We write $\zeta\one_{t_1}$ and $\zeta\two_{t_2}$ for normalised dilates of functions $\zeta\one$ on $G_1$ and $\zeta\two$ on $G_2$, and $\phitt$ and $\psitt$ for the product functions $\phi\one_{t_1} \otimes \phi\two_{t_2}$ and $\psi\one_{t_1} \otimes \psi\two_{t_2}$.

We define $\proT := \Rplus \times \Rplus$.
For $\prog :=(g_1,g_2)\in \proG$, $\prot \in \proT$ and $\eta \in [0,\infty)$, we denote by $P(\prog, \prot)$ the product of open balls $B_1(g_1, t_1)\times B_2(g_2, t_2)$ and by $\Gamma^{\eta }(\prog)$ the product cone $\Gamma_1^{\eta }(g_1)\times \Gamma_2^{\eta }(g_2)$, where
\[
\Gamma_i^{\eta }(g_i)
:=\{(h_i ,t_i )\in G_i \times \Rplus\colon\rho_i (g_i ,h_i ) \leq \eta t_i \}.
\]

For $\eta \in [0,\infty)$  and $f \in \Leb1(\proG)$, we define the maximal function:
\[
\maxfn{\phi}{\eta }(f)(\prog)
:=\sup \Bigl\{
\bigl| f \ast \phitt (\proh) \bigr| \colon
(\proh,\prot) \in \Gamma^\eta (\prog) \Bigr\}
\qquad\forall \prog\in\proG.
\]
This maximal function is called radial when $\eta = 0$ and nontangential when $\eta > 0$.

For $\eta \in \Rplus$  and $f \in \Leb1(\proG)$, we define the Lusin area function $\sqfn{\psi}{\eta }(f)$ by
\[
\sqfn{\psi}{\eta }(f)(\prog)
:=\lpar \iint_{\Gamma^{\eta }(\prog)}
\frac{\labs (f \ast \psitt(\proh) \rabs ^{2}}{\labs P(\prog,\eta \prot) \rabs  }
\wrt \proh \wrtprotonprot \rpar ^{1/2}
\qquad\forall \prog\in\proG.
\]
When $\eta = 0$  and $f \in \Leb1(\proG)$, we define the Littlewood--Paley square function $\sqfn{\psi}{0}(f)$ by
\[
\sqfn{\psi}{0}(f)(\prog)
:=\lpar \int_{\proT} \labs (f \ast \psitt(\prog) \rabs ^{2} \wrtprotonprot \rpar ^{1/2}
\qquad\forall \prog\in\proG.
\]
These are often given different symbols, but it is convenient to treat them together.

On a stratified Lie group $G$, we may take a basis $\{\oper{X}_1, \dots, \oper{X}_d\}$ for the space of left-invariant horizontal vector fields, and define the sublaplacian $\oper{L}$ to be $-\sum_{j-1}^{d} \oper{X}_j^2$.
The Riesz transformations are then the operators $\oper{X}_j \oper{L}^{-1/2}$.
The double Riesz transformations $\Riesz\one_{j_1} \otimes \Riesz\two_{j_2}$ on $\proG$ are defined in the obvious way when $1 \leq j_i \leq d_i$.

\begin{theorem}\label{thm main}
Let $\oper{T}$ be a maximal operator $\maxfn{\phi}{\eta }$ or a Littlewood--Paley operator $\sqfn{\psi}{\eta }$,  where $\eta \geq 0$, or a double Riesz transformation $\Riesz\one_{j_1} \otimes \Riesz\two_{j_2}$. 
Then 
\begin{align}\label{main LlogplusL}
\big| \big\{\prog \in  \proG\colon |\oper{T}f(\prog)|>\lambda\big\}\big| 
\lesssim   F_\Phi({f/\lambda})
\qquad\forall \lambda \in \Rplus,
\end{align}
for all $f \in \LlogplusL(\proG)$, 
where $\LlogplusL(\proG)$ is the Orlicz space associated to the functional $F_\Phi$, given by
\begin{align*}
F_{\Phi}(f) := \iint_{\proG} |f(\prog)| \log(\expe + |f(\prog)| ) \wrt \prog. 
\end{align*}
\end{theorem}

It is not possible to replace $F_\Phi(f/\lambda)$ by $\norm{f}_{\LlogplusL(\proG)}/\lambda$, where $\norm{\cdot}_{\LlogplusL(\proG)}$ is the Luxemburg norm associated to  $F_\Phi$.
We explain Orlicz spaces and Luxemburg norms later.
We say that the operator $\oper{T}$ is hyperweakly bounded when an estimate of the form \eqref{main LlogplusL} holds.

It is easy to iterate estimates for one-parameter maximal, area or singular integral operators to prove a local version of this result, where the support of the function $f$ is restricted to lie in a compact set.
However, iteration does not seem to be able to deal with the global case, and this is the main difficulty that we need to confront in this paper.

Our presentation has the following structure.
In Section 2 we review some background on stratified Lie groups, product spaces and Orlicz spaces.  

In Section 3 we show that the strong maximal operator $\strmaxfn$ is hyperweakly bounded, using a covering lemma, Theorem \ref{thm Covering}, that goes back to \cite{CF}.
We then prove Theorem \ref{thm main} for the maximal operator $\maxfn{\phi}{\eta }$, by using group properties to dominate the maximal operator $\maxfn{\phi}{\eta }$ by the strong maximal operator. 

In Section 4, we construct the atomic decomposition.
We use the gradient $\fullnabla \pois$ of the Poisson kernel $\pois$ as our auxiliary function, and then have a key result of \cite[(12) and following]{CFLY}, namely, the good-$\lambda$ inequality:
\begin{align}\label{GoodLambda}
\bigl|\bigl\{\prog \in \proG\colon \sqfn{\fullnabla \pois}{1}(f)(\prog) > \lambda \bigr\} \bigr|
&\lesssim \bigl| L_{\eta }^c \bigr|
+
\frac{1}{\lambda^2} \int_{L_{1}(\lambda )} \maxfn{\pois}{\eta }^2 (f)(\prog) \wrt \prog,
\end{align}
when $\eta $ is sufficiently large; here $L_{\eta }(\lambda) :=\lset \prog\in {\proG}\colon \maxfn{\pois}{\eta }(f)(\prog) \leq \lambda  \rset$.
This implies that the area operator $\sqfn{\fullnabla \pois}{1}$ is hyperweakly bounded. 
By using this boundedness and the Calderón reproducing formula, we can decompose global $\LlogplusL$ functions into atoms.

In Section 5, we apply our atomic decomposition and a version of Journé's covering lemma for spaces of homogeneous type established in \cite{HLL}.
We prove Theorem \ref{thm main} for \emph{general} operators $\sqfn{\psi}{\eta }$  and the double Riesz transformations $\Riesz\one_{j_1} \otimes \Riesz\two_{j_2}$. 
The same argument holds for general product Calderón--Zygmund operators, as in Journé \cite{Jo}.

Most of the arguments rely only the theory of spaces of homogeneous type. 
However, we need the setting of a stratified Lie group in two places. 
First, it gives us the good $\lambda$ inequality \eqref{GoodLambda}. 
Second, it gives us the Calderón reproducing formula for $\Leb 2(G)$ functions, which is needed for the atomic decomposition. 

``Constants'' are positive real numbers, depending only on the geometry of $\proG$ unless otherwise indicated; we write $A \lesssim B$ when there exists a constant $C$ such that $A \leq C B$.
We write $\chi_{E} $ for the indicator function of a set $E$, and $o$ denotes the identity of a group.

\section{Preliminaries}

\subsection{Stratified nilpotent Lie groups}\label{ssec:strat-groups}
Let $G$ be a (real and finite dimensional) stratified nilpotent Lie group of step $s$ with Lie algebra $\Lie{g}$.
This means that we may write $\Lie{g}$ as a vector space direct sum $\bigoplus_{j = 1}^{s} \Lie{v}_{j} $, where $[\Lie{v}_1, \Lie{v}_{j}] = \Lie{v}_{j + 1}$ when $1\leq j \leq s$; here $\Lie{v}_{s+1} := \{0\}$.
Let $\homodim$ denote the homogeneous dimension of $G$; that is, $\sum_{j=1}^{s} j \dim \Lie{v}_{j}$.
There is a one-parameter family of automorphic dilations $\delta_t$ on $\Lie{g}$, given by
\begin{align*}
\delta_t (\oper{X}_1 + \oper{X}_2+ \dots + \oper{X}_{s}) := t\oper{X}_1 + t^2\oper{X}_2 + \dots + t^{s} \oper{X}_{s};
\end{align*}
here each $\oper{X}_j \in \Lie{v}_{j}$ and $t \in \Rplus$.
The exponential mapping $\exp\colon \Lie{g} \to G$ is a diffeomorphism, and we identify $\Lie{g}$ and $G$.
The dilations extend to automorphic dilations of $G$, also denoted by $\delta_t$, by conjugation with $\exp$.
The Haar measure on $G$, which is bi-invariant, is the Lebesgue measure on $\Lie{g}$ lifted to $G$ using $\exp$.

By \cite{HS}, the group $G$ may be equipped with a smooth subadditive homogeneous norm $\rho$, a continuous function from $G$ to $[0,\infty)$ that is smooth on $G\setminus \{o\}$
and satisfies
\begin{enumerate}
\item $\rho(g^{-1}) =\rho(g)$;
\item $\rho(g^{-1}h) \leq \rho(g) + \rho(h)$;
\item $\rho({ \delta_t(g)}) =t\rho(g)$ for all $g\in G$ and all $t \in \Rplus$;
\item $\rho(g) =0$ if and only if $g=o$.
\end{enumerate}
Abusing notation, we set $\rho(g, g') := \rho(g^{-1} g')$
for all $g, g' \in G$; this defines a metric on $G$.
We write $B(g, r)$ for the open ball with centre $g$ and radius $r$ with respect to $\rho$:
\[
B(g, r) = g B(o,r) = g \{ h \in G \colon \rho(h) < 1 \}.
\]
The metric space $(G,\rho)$ is \emph{geometrically doubling}; that is, there exists $A \in \N$ such that every metric ball $B(x,2r)$ may be covered by at most $A$ balls of radius $r$.

We remind the reader that a stratified Lie group is a space of homogenous type in the sense of Coifman and Weiss \cite{CW1, CW2}, and analysis on stratified Lie groups uses much from the theory of such spaces.
In particular, we frequently deal with \emph{molecules}, that is, functions $\zeta$ that satisfy \emph{standard decay and smoothness conditions}, by which we mean that there is a parameter $\epsilon \in (0,1]$, which we fix once and for all, such that
\begin{equation}\label{eq:molecule}
\begin{gathered}
\labs \zeta(g) \rabs  \lesssim \frac{1}{(1+ \rho(g))^{\homodim+\epsilon}}, \\
\labs \zeta(g)-\zeta(g') \rabs
\lesssim \frac{\rho( g' g^{-1} )^\epsilon}{(1+ \rho(g)+ \rho(g'))^{\homodim+2\epsilon} }
\end{gathered}
\end{equation}
for all $g, g' \in G$.
We often impose an additional \emph{cancellation condition}, namely
\begin{gather}\label{cancel psi}
\int_G \zeta(g) \wrt g =0 .
\end{gather}

The \emph{normalised dilate} $f_t$ of a function $f$ on $G$ by $t \in \Rplus$ is given by $f_{t} := t^{-\homodim}f\circ \delta_{1/t}$, and the \emph{convolution} $f\ast f'$ of suitable functions $f$ and $f'$ on $G$ is defined by
\begin{align*}
f \ast f'(g)
:=\int_Gf(h)f'(h^{-1}g)\wrt h
=\int_Gf(gh^{-1})f'(h)\wrt h.
\end{align*}

Take left-invariant vector fields $\oper{X}_1$, \dots, $\oper{X}_{n}$ on $G$ that form a basis of $\Lie{v}_1$, and define the \emph{sub-Laplacian} $\Lap := -\sum_{j=1 }^{n} (\oper{X}_{j})^2 $.
Observe that each $\oper{X}_{j}$ is homogeneous of degree $1$ and $\Lap$ is homogeneous of degree $2$, in the sense that
\begin{align*}
&\oper{X}_{j} \lpar  f \circ \delta_{t} \rpar
= t \lpar  \oper{X}_{j} f \rpar  \circ \delta_{t}, \\
&\Lap \lpar  f \circ \delta_{t} \rpar
= t^2  \lpar  \Lap f \rpar  \circ \delta_{t}
\end{align*}
for all $t \in \Rplus$ and all $f \in \fnspace{C}^{2}(G)$.

Associated to the sub-Laplacian, Folland \cite{F2} defined the Riesz potential operators $\Lap^{-\alpha}$, where $\alpha \in \Rplus$; these are convolution operators with homogeneous kernels.
The \emph{Riesz trans\-formation} $\Riesz_{j} := \oper{X}_{j} \Lap^{-1/2}$ is a singular integral operator, and is bounded on $\Leb{p}(G)$ when $1 < p < \infty$ as well as from the Folland--Stein Hardy space $\Hardy(G)$ to $\Leb1(G)$.

The Hardy--Littlewood maximal operator $\oper{M}$ on $G$ is defined using the metric balls:
\[
\oper{M}f(g)
:= \sup\lset   \dashint_{B(g',r)} \labs f(g'') \rabs  \wrt g'' \colon g \in B(g',r) \rset .
\]
where the ``average integral'' $\dashint$ is defined by
\[
\dashint_B f(\prog) \wrt\prog := \frac{1}{\labs B \rabs } \int_B f(\prog) \wrt\prog.
\]
For future use, we note that the layer cake formula implies that, if $\mu$ is a radial decreasing function on $G$ (that is, $\mu(g)$ depends only on $\rho(g)$ and decreases as $\rho(g)$ increases), then
\begin{equation}\label{eq:decay-max}
\labs f \rabs  \ast \mu_\epsilon (g)
\leq  \norm{ \mu}_{\Leb1(G)} \oper{M}f(g)
\qquad\forall g \in G.
\end{equation}

\subsection{Functional calculus for the sub-Laplacian}\label{ssec:funct-calc}
The sub-Laplacian $\Lap$ has a spectral resolution:
\[
\Lap (f)=\int_{\Rplus}\lambda \wrt \oper{E}_{\Lap }(\lambda) f
\qquad\forall f\in \Leb2(G),
\]
where $\oper{E}_{\Lap}(\lambda)$ is a projection-valued measure on $[0,\infty)$, the spectrum of $\Lap$.
For a bounded Borel function $m\colon[0,\infty)\to \C$, we define the operator $F(\Lap)$ spectrally:
\begin{equation*}
m(\Lap)f := \int_{\Rplus} m(\lambda)\wrt \oper{E}_{\Lap}(\lambda) f
\qquad\forall f\in \Leb2(G).
\end{equation*}
This operator is a convolution with a Schwartz distribution on $G$.

\subsection{The heat and Poisson kernels}
Let $\heatt$ and $\poist$, where $t \in \Rplus$, be the heat and Poisson kernels associated to the sub-Laplacian operator $\Lap$, that is, the convolution kernels of the operators $e^{t\Lap}$ and $e^{t \sqrt{\Lap}}$ on $G$.
We write $\qt$ for $t\partial_t \poist$.
We warn the reader that $\poist$ and $\qt$ are the normalised dilates of $\poisone$ and $\qone$ by the factor $t$, but $\heatt$ is the normalised dilate of $\heatone$ by a factor of $t^{1/2}$.
Let $\nabla$ denote the subgradient on $G$ and $\fullnabla$ denote the gradient $(\nabla , \partial_{t})$ on $G\times\Rplus$.

\begin{lemma}\label{lem:heat-pois-estim}
The kernels $\heatt$ and $\poist$ are $\Rplus$-valued.
Further, $\heatt$ and $\poist$ have integral $1$, while $\qt$ has integral $0$
for all $t \in \Rplus$.
Finally, there exists a constant $c$ such that
\begin{align*}
\heatt(g)
&\lesssim t^{-\homodim/2}\exp\lpar -\sfrac{\rho^{2}(g)}{ct}\rpar,
\\
\labs \fullnabla\heatt(g) \rabs
&\lesssim t^{-(\homodim+1)/2}\exp\lpar -\sfrac{\rho^{2}(g)}{ct}\rpar,
\\
\poist(g)
&\eqsim \frac{ t }{ (t^2 + \rho(g)^2)^{(\homodim + 1)/2}},
\\
\labs \fullnabla \poist(g) \rabs
&\lesssim \frac{ t }{ (t^2 + \rho(g)^2)^{(\homodim + 2)/2}}
\end{align*}
for all $g\in G$ and $t \in \Rplus$.
\end{lemma}
\begin{proof}
For the heat kernel estimates, see \cite[Theorem IV.4.2]{VSC92}.
Note that the first estimate has a version with the opposite inequality, with a different constant $c$.

The estimates for $\poist$ and $\qt$ follow from the subordination formula
\begin{equation*}
e^{-t\sqrt{\Lap}}
=\frac{1}{2\sqrt{\pi}}\int_{\Rplus}
\frac{te^{-{t^2}/{4v}}}{\sqrt{v}}e^{-v \Lap}\frac{\wrt v}{v}.
\end{equation*}
We leave the details to the reader.
\end{proof}

This lemma shows that the heat kernel $\heatone$ and the Poisson kernel $\poisone$ and their derivatives satisfy the standard decay and smoothness conditions \eqref{eq:molecule}; their derivatives also satisfy the cancellation condition \eqref{cancel psi}.

%
%

\subsection{Systems of pseudodyadic cubes}
\label{sec:pseudodyadiccubes}

We use the Hytönen--Kairema \cite{HK} families of ``dyadic cubes" in geometrically doubling metric spaces.
We state a version of \cite[Theorem 2.2]{HK} that is simpler, in that we work on metric spaces rather than pseudo\-metric spaces.
The Hytönen--Kairema construction builds on seminal work of Christ \cite{Chr} and of Sawyer and Wheeden \cite{SW}.

\begin{theorem}[\cite{HK}]\label{thm:Hyt-Kai}
\label{thm:pseudodyadiccubes}
Let $c$, $C$ and $\kappa$ be constants such that $0 < c \leq C < \infty$ and $12 C\kappa \leq c$, and let $(G,\rho)$ be a metric stratified group. 
Then, for all $k \in \Z$, there exist families $\dcubes_k(G)$ of \emph{pseudo\-dyadic cubes} $Q$ with \emph{centres} $z(Q)$, such that:
\begin{enumerate}
\item $G$ is the disjoint union of all $Q \in \dcubes_k(G)$, for each $k\in\Z$;
\item $B(z(Q),c\kappa^k/3)\subseteq Q \subseteq B(z(Q),2C\kappa^k)$ for all $Q \in \dcubes_k(G)$;
\item if $Q \in \dcubes_k(G)$ and $Q' \in \dcubes_{k'}(G)$ where $k\leq k'$, then either $Q \cap Q'=\emptyset$ or $Q \subseteq Q'$; in the second case, $B(z(Q), 2C\kappa^k) \subseteq B(z(Q'),2C\kappa^{k'})$;
\end{enumerate}
\end{theorem}

We write $\dcubes(G)$ for the union of all $\dcubes_k(G)$, and call this a system of pseudo\-dyadic cubes.
Given a cube $Q \in \dcubes_k(G)$, we denote the quantity $\kappa^k$ by $\len(Q)$, by analogy with the side-length of a Euclidean cube.

A finite collection $\{\dcubes^{\indext}\colon {\indext}=1,2,\dots ,\Indext\}$ of systems of pseudodyadic
cubes is called a \emph{collection of adjacent systems of pseudo\-dyadic cubes with parameters $C'$, $c$, $C$ and $\kappa$}, if it has the following properties: individually, each $\dcubes^{\indext}$ is a
system of pseudo\-dyadic cubes with parameters $c$, $C$ and $\kappa$ as in Theorem \ref{thm:Hyt-Kai}; collectively, for each ball $B(x,r)\subseteq G$ such that $\kappa^{k+3}<r\leq\kappa^{k+2}$, where $k\in\Z$, there exist ${\indext} \in \{1, 2, \dots, \Indext\}$ and $Q\in\dcubes^{\indext}_k$ with centre $z(Q)$ such that $d(x,{}^{\indext}x_\alpha^k) < 2\kappa^{k}$ 
\begin{equation}\label{eq:ball;included}
    B(x,r)\subseteq Q\subseteq B(x,C'r).
\end{equation}

The following construction is due to \cite{HK}.

\begin{theorem}\label{thm:existence2}
Suppose that $(G,\rho)$ is a metric stratified group.
Then there exists a finite collection $\{\dcubes^{\indext}\colon {\indext} = 1,2,\dots ,\Indext\}$ of adjacent systems of pseudo\-dyadic cubes with parameters $C'$, $c$, $C$ and $\kappa$, where $\kappa := 1/100$, $c := 12^{-1}$, $C := 4$ and $C' := 8 \times 10^{6}$. 
For each ${\indext}\in\{1,2,\dots,\Indext\}$, the centres $z(Q)$ of the cubes $Q \in\dcubes^{\indext}_k$ have the  properties that
\[
d(z(Q), z(Q')) \geq \kappa^k /4
\qquad\text{when $Q\neq Q'$}
\]
and
\[
\min \{ d(x, z(Q)) \colon Q \in \dcubes^{\indext}_k \} < 2\kappa^k
\qquad\forall x\in G.
\]
\end{theorem}

From \cite[Remark 2.8]{KLPW}, the number $\Indext$ of adjacent systems of pseudo\-dyadic cubes in Theorem \ref{thm:existence2} may be taken to be at most $A^6 \kappa^{-\log_2(A)}$, where $A$ is the geometric doubling constant of $G$.
The constants $c$, $C$, $C'$ and $\kappa$ do not depend on the choice of the metric stratified group $(G,\rho)$.

\subsection{Products of stratified groups}

We equip the product of two stratified groups $G_1$ and $G_2$ with a product structure.
We carry forward the notation from Section \ref{ssec:strat-groups}, adding a subscript $i$ or superscript $[i]$ to clarify that we are dealing with $G_i$; the parameter $i$ is always $1$ or $2$.
To shorten the formulae, we often use bold face type to indicate a product object: thus we write $\proG$, $\proT$, $\prog$, $\pror$ and $\prot$ in place of $G_1 \times G_2$, $\Rplus \times \Rplus$, $(g_1,g_2)$, $(r_1,r_2)$ and $(t_1,t_2)$.
For example, $B_i (g_i, r_i)$ denotes the open ball in $G_i$ with centre $g_i$ and radius $r_i$, with respect to the homogeneous norm $\rho_i $, and we write $P(\prog,\pror)$ for the product $B_1(g_1, r_1) \times B_2(g_2, r_2)$.
We write $\pi_i$ for the projection of $\proG$ onto $G_i$.

Products of balls are basic geometric objects and we write $\rect(\proG)$ for the family of all such products.
In addition we deal with \emph{rectangles}, by which we mean products of pseudo\-dyadic cubes.
We write $\drect(\proG)$ for the family of all rectangles, and for adjacent pseudo\-dyadic systems as constructed in Theorem \ref{thm:existence2}, we let $\drect^{{\indext}_1,{\indext}_2}(\proG)$ be the family of rectangles $\{Q_1\times Q_2\colon Q_i\in \dcubes^{{\indext}_i},\ {\indext}_i=1,2,\dots ,\Indext_i\}$.
We let $\blen\colon \drect(\proG) \to \proT$ be the function such that $\len_i(Q_1 \times Q_2) = \len(Q_i)$, the ``side-length'' of $Q_i$.

The element of Haar measure on $\proG$ is denoted $\wrt\prog$, but is often written $\wrt g_1\wrt g_2$ for calculations.
The convolution $f\ast f'$ of suitable functions $f$ and $f'$ on $\proG$ is defined by
\begin{align*}
(f\ast f')(\prog)
:=\int_{\proG}f(\proh)f'(\proh^{-1}\prog) \wrt \proh.
\end{align*}

The strong maximal operator $\strmaxfn$ is defined by
\begin{align*}
\strmaxfn(f)(\prog)
:=\sup \lset \dashint_{P}\labs f(\proh) \rabs \wrt \proh \colon P \in \rect(\proG), P \ni \prog \rset.
\end{align*}
for all $f \in \Leb1_{\loc}(\proG)$.
It is evident that $\strmaxfn$ is dominated by the composition of the Hardy--Littlewood maximal operators in the factors:
\[
\strmaxfn{f}
\leq \oper{M}_1 \oper{M}_2 (f)
\qquad\text{and}\qquad
\strmaxfn{f}
\leq \oper{M}_2 \oper{M}_1(f).
\]
When $1 < p \leq \infty$, the operators $\oper{M}_1$ and $ \oper{M}_2$ in the factors are $\Leb{p}$-bounded, so the iterated maximal operators and hence the strong maximal operator are also $\Leb{p}$-bounded.

Given an open subset $U$ of $\proG$ with finite measure $\labs U \rabs $, we define the enlargement $U^{*}$ of $U$ using the strong maximal operator $\strmaxfn$:
\begin{align*}
U^{*}
:= \Bigl\{ \prog \in \proG \colon \strmaxfn \chi_{U}(\prog) > \alpha \Bigr\} ,
\end{align*}
for some $\alpha \in (0,1)$ that varies from instance to instance.
We write $\maxrect(U)$ for the family of maximal rectangles contained in $U$.

If $\phi\one$ on $G_1$ and $\phi\two$ on $G_2$ both satisfy the decay and smoothness conditions \eqref{eq:molecule} and have integral $1$, then
\begin{equation}\label{eq:Poisson-leq-strong}
\labs  f \ast \phitt(\prog)  \rabs  \lesssim \strmaxfn(f)(\prog)
\qquad\forall \prog \in \proG \quad\forall f \in \Leb1(\proG),
\end{equation}
much as argued to prove \eqref{eq:decay-max}, but with ``biradial'' in place of ``radial''.

\subsection{Journé's covering lemma}

We recall a covering lemma for product spaces. 
Journé  \cite{Jo} first established this result on $\R\times\R$. 
The fourth author  \cite{P} extended it to higher dimensional Euclidean spaces and an arbitrary number of factors $\R^{n_1}\times \R^{n_2}\times \cdots\times \R^{n_m}$. 
This was extended to products of spaces of homogeneous type in \cite[Lemma 2.1]{HLL}.

Let $U$ be an open subset of $\proG$ of finite measure and $\maxrect_i(U)$ denote the family of rectangles $R$ in $U$ which are maximal (in terms of inclusion) in the $i$th ``direction''. 
Recall that $\maxrect(U)$ is the set of all maximal subrectangles of $U$.  
Given $R=Q_1\times Q_2\in \maxrect_1(U)$, let $\tilde{Q}_2$ be the biggest pseudo\-dyadic cube containing $Q_2$ such that
\[
 \big|\biglpar Q_1\times \tilde{Q}_2\bigrpar \cap U\big|>\frac{1}{2}|Q_1\times \tilde{Q}_2|.
\]
Similarly, given $R=Q_1\times Q_2\in \maxrect_2(U)$, let $\tilde{Q}_1$ be the biggest pseudo\-dyadic cube
containing $Q_1$ such that
\[
 \big|\biglpar \tilde{Q}_1\times Q_2\bigrpar \cap U\big|>\frac{1}{2}|\tilde{Q}_1\times Q_2|.
\]
We then set 
\[
 \gamma_1(R) := \frac{\len(\tilde{Q}_1)}{\len(Q_1)}
\qquad\text{and}\qquad 
\gamma_2(R) := \frac{\len(\tilde{Q}_2)}{\len(Q_2)}.
\]

Here is our Journé-type covering lemma on $\proG$.

\begin{lemma}\label{theorem-cover lemma}
Suppose that $U$ is an open subset in $\proG$ of finite measure and $\delta\in \Rplus$. 
Then 
\begin{align*}
\sum_{R\in \maxrect_1(U)}\abs{R} \gamma_1(R)^{-\delta}
\lesssim |U|
\qquad\text{and}\qquad
\sum_{R\in \maxrect_2(U)}\abs{R} \gamma_2(R)^{-\delta}
\lesssim |U|.
\end{align*}
\end{lemma}

\begin{proof}
Apply the statement in \cite[Lemma 2.1]{HLL} to $\proG$.
\end{proof}

\subsection{The Orlicz space $\LlogplusL(\proG)$}
We recall the definition of an Orlicz space on $\proG$. 
A Young function is a continuous, convex, increasing bijective function $\Phi\colon[0,\infty)\to [0,\infty)$.
To a Young function $\Phi$, we associate the (usually nonlinear) functional $F_\Phi$, given by
\[
F_{\Phi}(f) = \int_{\proG} \Phi(|f(\prog)|) \wrt \prog
\qquad\forall f \in \Leb1_{\loc}(\proG).
\]
The convexity of $\Phi$ implies that $\{ f \in \Leb1_{\loc}(\proG) \colon F_\Phi(f) \leq 1 \}$ is a closed convex symmetric set, and so the Luxemburg norm, given by
\[
 \|f\|_{\Leb{\Phi}(\proG)} := \inf\lset \lambda\in \Rplus \colon F_{\Phi}(f/\lambda) \leq 1 \rset,
\]
is indeed a norm, and $\Leb\Phi(\proG)$, the set of functions $f$ for which $\|f\|_{\Leb{\Phi}(\proG)}$ is finite, is a Banach space. 
The sets $\{ f \in \Leb1_{\loc}(\proG) \colon F_\Phi(f) \leq 1 \}$ and $\{ f \in \Leb1_{\loc}(\proG) \colon \norm{f}_{\Leb{\Phi}(\proG)} \leq 1 \}$ coincide.

Suppose that $\Phi$ and $\Psi$ are Young functions such that
\begin{equation}\label{eq:conjugate}
st \leq \Phi(s)+\Psi(t)
\qquad\forall s,t \in \Rplus.
\end{equation}
Then
\begin{equation}\label{eq:pre-Holder}
\labs \int_\proG f(\prog) h(\prog) \wrt \prog \rabs
\leq \int_\proG \Phi(f(\prog)) \wrt \prog + \int_\proG \Psi( h(\prog)) \wrt \prog ,
\end{equation}
and it follows that the corresponding Luxemburg norms are related by the \emph{generalized Hölder inequalities}, namely,
\begin{equation}\label{eq:general-Holder}
\int_\proG f(\prog) h(\prog) \wrt \prog
\leq 2\|f\|_{\Leb\Phi(\proG)} \|h\|_{\Leb{\bar\Phi}(\proG)}.
\end{equation}

We are particularly interested in a special pair of Young functions.
Henceforth,
\begin{equation}\label{def:Phi-and-Psi}
 \Phi(s):=s[\log(\expe + s)]\qquad\text{and}\qquad \Psi(t) := \exp(t)-1 .
\end{equation}
By maximising in $s$, it is straightforward to show that $1 + st - s \log(\expe+s) \leq \expe^t$, so these functions satisfy \eqref{eq:conjugate}, whence \eqref{eq:pre-Holder} and \eqref{eq:general-Holder} hold.
For this choice of $\Phi$ and $\Psi$, the corresponding spaces are denoted by $\LlogplusL(\proG)$ and $\expe^{\Leb{}}(\proG)$.
Clearly $\LlogplusL(\proG) \subseteq \Leb{1}(\proG)$.
Further, $\Phi(\lambda t) \leq \Phi(\lambda) \Phi(t)$, whence
\begin{equation}\label{eq:nearly-linear}
F_\Phi(\lambda f) \leq \Phi(\lambda) F_\Phi(f)
\qquad\forall \lambda\in \Rplus \quad\forall f \in \LlogplusL(\proG).
\end{equation}
This implies that $f \in \LlogplusL(\proG)$ if and only if $F_\Phi(f)$ is finite.
Further, the hyperweak boundedness estimate \eqref{main LlogplusL} that we work with is equivalent to an estimate of the form
\begin{align*}
\big| \big\{\prog \in  \proG\colon |\oper{T}f(\prog)|>\lambda\big\}\big| 
\lesssim   \frac{\log(\expe + 1/\lambda)}{\lambda} \norm{f}_{\LlogplusL(\proG)}
\qquad\forall \lambda \in \Rplus,
\end{align*}
for all $f \in \LlogplusL(\proG)$ of norm $1$, that is, $\oper{T}f$ decays a little slower at infinity that would occur if we had a weak type estimate.
Finally, since $\Psi$ is a Young function,
\begin{equation}\label{eq:Jensen}
\Psi(t/\lambda) \leq \Psi(t)/\lambda
\qquad\forall \lambda\in [1,\infty) .
\end{equation}

We need a density result. 

\begin{proposition}\label{prop density}
$\Leb 2 \cap \LlogplusL (\proG)$ is dense in $\LlogplusL (\proG)$.
\end{proposition}

\begin{proof}
For $f\in \LlogplusL (\proG)$, we define, for all $N \in \N$,
\[
  f_N := \chi_{\{\prog\in \proG\colon |f(\prog)| \leq N\}} f .
\]
Then $f_N\in \LlogplusL (\proG)$, $|f_N| \leq |f|$, and $f_N \to f$ almost everywhere as $n \to \infty$. 
Next,
\begin{align*}
\int_{\proG} |f_N(\prog)|^2 \wrt \prog 
\leq \int_{\proG} N |f(\prog)| \log(\expe + |f(\prog)|) \wrt \prog
= N F_{\Phi}(f),
\end{align*}
which shows that $f_N \in \Leb 2(\proG)$ for all $N \in \Z^+$.
Finally, if $\lambda < 1$, then
\begin{align*}
F_{\Phi}((f-f_N)/\lambda)
&= \frac{1}{\lambda} \int_{\proG} |f(\prog)-f_N(\prog)| \log(\expe + |(f(\prog)-f_N(\prog))/\lambda)| \wrt \prog\\
&\leq \frac{1}{\lambda} \int_{\{\prog\in \proG\colon |f(\prog)|>N\}} |f(\prog)|
\lpar \log(\expe + |f(\prog)|) + \log(1/\lambda) \rpar \wrt \prog \\
&\to 0
\end{align*}
as $N \to \infty$.
Consequently, $\lnorm{f-f_N}\rnorm_{\LlogplusL(\proG)}$ tends to zero as $N$ tend to infinity.
\end{proof}

\section{The strong maximal function and Proof of Theorem \ref{thm main}}

Before we tackle the main topic of this section, we state and prove a covering lemma.
We begin with a geometric observation.

\begin{lemma}\label{lem:compare-R-P}
There is a geometric constant $C(\proG)$ with the property that
\[
R \subseteq \{ g\in \proG \colon \strmaxfn(\chi_U)(\prog) > C(\proG) \alpha \}
\]
for all measurable subsets $U$ of $\proG$, all $\alpha \in (0,1)$ and all rectangles $R \in \rect$ such that
\[
|R \cap U| \geq \alpha \abs{R}.
\]
\end{lemma}

\begin{proof}
By definition, $R$ is a product of pseudodyadic cubes $Q_1 \times Q_2$, and by Theorem \ref{thm:Hyt-Kai},
\[
B(z(Q_i),c \ell(Q_i)/3)\subseteq Q_i \subseteq B(z(Q_i),2C\ell(Q_i)) .
\]
Define $P := B(z(Q_1),2C\ell_1(R)) \times B(z(Q_2),2C\ell_2(R))$.
Then $R \subseteq P$ and
\begin{align*}
|P \cap U| 
&\geq |R \cap U|  
\geq \lambda \abs{R}  \\
&\geq \lambda  | B(z(Q_1),c\ell(Q_1)/3) \times B(z(Q_2),c\ell(Q_2)/3)|  \\
&= \lambda   \Bigl(\frac{c }{6C}\Bigr)^{\homodim_1+\homodim_2}  
| B(z(Q_i),2C\ell(Q_i)) \times B(z(Q_2),2C\ell(Q_2))| \\
&=: \lambda C(\proG) |P|.
\end{align*}
Hence $R$ is a subset of the set 
\[
U^* = \{ \prog \in \proG \colon \strmaxfn \chi_{U}(\prog) > C(\proG) \lambda \},
\] 
as required.
\end{proof}

\begin{theorem}{\cite{CF}}\label{thm Covering}
Let $\{R_j\}_{j\in J}$ be a family of rectangles in $ \proG$ such that $\big|\bigcup_{j\in J} R_j\big|$ is finite.  
Then there is a sequence of rectangles $\{\tilde{R}_k\}\subset \{R_j\}_{j\in J}$ such that
\[
\big|\bigcup_{j\in J} R_j\big|\lesssim \big|\bigcup_{k} \tilde{R}_k\big| 
\qquad\text{and}\qquad
\Biglnorm \sum_k\chi_{\tilde{R}_k} \Bigrnorm_{\expe^{\Leb{}}(\proG)} \lesssim \big|\bigcup_{j\in J} R_j\big|.
\]
\end{theorem}

\begin{proof}
The proof is a generalisation of that in \cite{CF}.
However, the pseudo\-dyadic case is some\-what trickier than the dyadic case in $\R^2$, and the argument in \cite{CF} is rather brief and not always precise, so it seems worthwhile to provide a complete proof.

Since there are countably many rectangles, we may assume that $J = \N$ and $j \in \N$.
Let $E$ be the set $\bigcup_j R_j$.

Choose a subsequence $\{R_{\sigma(j)} \colon j \in \N\}$ of $\{R_j \colon j \in \N\}$, using the rules that $\sigma(1) = 1$ and $\sigma(k+1)$ is the least $j > \sigma(k)$ such that
\[
| R_j \cap (R_{\sigma(1)} \cup \dots \cup R_{\sigma(k)}) | \leq \frac{1}{4} |R_j|;
\]
the construction terminates if this is not possible. 
Let $E_K = \bigcup_{1 \leq k \leq K} R_{\sigma(k)} \subseteq E$.
If the construction does not terminate, then by monotone convergence, we may choose $K$ such that $|E_K| \geq |E_\infty|/2$; if the construction terminates, we take $K$ to be the last index of the finite sequence.
The set $E$ consists of rectangles $R_j$ that are not one of the chosen rectangles $R_{\sigma(k)}$ together with rectangles that are subsets of $E_\infty$ if the construction did not terminate, or of $E_K$ otherwise.
Suppose that the construction did not terminate.
Then by Lemma \ref{lem:compare-R-P}, 
\[
R_j \subseteq E_\infty^* := \{ \prog \in \proG \colon \strmaxfn \chi_{E_\infty}(\prog) > {C(\proG)}/{4} \}.
\] 
Hence $E \subseteq E_\infty^* \cup E_\infty$, and the boundedness of $\strmaxfn$ on $\Leb2(\proG)$ shows that 
\[
|E| \leq |E_\infty^*| + |E_\infty| \lesssim |E_\infty| \leq 2 |E_K|.
\]
When the construction terminates, a slightly easier argument shows that $|E| \lesssim |E_K|$.

We now relabel the finite family of rectangles $\{R_{\sigma(1)}, \dots, R_{\sigma(K)}\}$ as $R_K, \dots, R_1$ and repeat the construction, setting $\tau(1) := 1$ and choosing $\tau(l+1)$ to be the least $l > \tau(l)$ such that
 \[
| R_k \cap (R_{\tau(1)} \cup \dots \cup R_{\tau(l)}) | \leq \frac{1}{4} |R_k|;
\]
we end up with $L$ terms, say.
Since
 \[
| R_k \cap (R_{k+1} \cup \dots \cup R_{K}) | \leq \frac{1}{4} |R_k|
\]
by construction, it follows that
\begin{equation}\label{eq:disjointness}
| R_{\tau(l+1)} \cap ((R_{\tau(1)} \cup \dots \cup R_{\tau(l)}) \cup (R_{\tau(l+2)} \cup \dots \cup R_{\tau(L)}))| \leq \frac{1}{2}  | R_{\tau(l+1)} |.
\end{equation}

Write $\tilde{R}_l$ for $R_{\tau(l)}$ and $\tilde{E}$ for $\bigcup_{l} R_{\tau(l)}$.
Much as before, $|E| \lesssim \biglabs\tilde{E}\bigrabs$.
Note that
\begin{equation}\label{eq:measure-union-like-sum-measures}
\biglabs \tilde{E}\bigrabs 
\leq \sum_{l} \biglabs \tilde{R}_l \bigrabs
\leq 2 \sum_{l} \bigglabs \tilde{R}_l \setminus \bigglpar \bigcup_{l' \neq l} \tilde{R}_{l'} \biggrpar \biggrabs
= 2 \bigglabs \bigcup_{l} \tilde{R}_l \setminus \bigglpar \bigcup_{l' \neq l} \tilde{R}_{l'} \biggrpar \biggrabs
\leq 2 \biglabs \tilde{E} \bigrabs .
\end{equation}

We now claim that
\begin{equation}\label{eq:claim}
\biglabs \{ \prog \in \proG \colon \sum_{l} \chi_{\tilde{R}_l} \geq n \} \bigrabs
\lesssim 2^{-n/2} |E|.
\end{equation}
From this claim, it follows that
\[
\begin{aligned}
\int_{\proG } \Psi\Biglpar \sum_{l} \chi_{\tilde{R}_L}(\prog) / \lambda \Bigrpar \wrt\prog 
&\lesssim \sum_{n=1}^{\infty}  \Psi(n/\lambda) 2^{-n/2} |E|   \\
&\lesssim \sum_{n=1}^{\infty}  \exp(n/\lambda) 2^{-n/2} |E|   \\
&= \sum_{n=1}^{\infty} \exp( n[ 1/\lambda - \log(2)/2]) |E| \\
&\lesssim |E|
\end{aligned}
\]
as long as $1/\lambda < \log(2)/2$, and so $\sum_{l=1}^{L} \chi_{R_l} \in \expe^{\Leb{}}(G)$.
It remains to prove \eqref{eq:claim}.

Consider $\prog \in \proG$ that lies in two distinct rectangles from the family $\{ \tilde{R}_l\}$, $R$ and $S$ say.
If $\len_1(S) = \len_1(R)$, then $\pi_1(R) = \pi_1(S)$, since both $\pi_1(R)$ and $\pi_1(S)$ are pseudo\-dyadic cubes containing $\pi_1(\prog)$; now $\pi_2(R) \subseteq \pi_2(S)$ or $\pi_2(R) \supseteq \pi_2(S)$, so $R\cap S$ is either $R$ or $S$, and this contradicts \eqref{eq:disjointness}.
Hence we may assume, without loss of generality, that $\len_1(S) < \len_1(R)$.
A similar argument then shows that $\len_2(S) > \len_2(R)$.
Thus, if there are $n$ distinct rectangles in $\{ \tilde{R}_l\}$ that contain $\prog$, then we may label them $R_1(\prog)$, \dots, $R_n(\prog)$, in such a way that $\len_1(R_j(\prog))$ decreases with $j$ and $\len_2(R_j(\prog))$ increases with $j$; then \[
R_1(\prog) \cap \dots \cap R_n(\prog) = \pi_1(R_n(\prog)) \times \pi_2(R_1(\prog)),
\] 
and
\[
\labs R_1(\prog) \cap \dots \cap R_n(\prog) \rabs 
= \labs \pi_1(R_n(\prog)) \times \pi_2(R_1(\prog)) \rabs
\lesssim \kappa ^{-n \homodim_1} \labs R_1(\prog) \rabs .
\]

We say that a rectangle $T \in \{ \tilde{R}_l \}$ is a \emph{descendant} of a rectangle $R \in \{ \tilde{R}_l\}$, and we write $T \succ R$, if $R\cap T \neq \emptyset$ and both $\len_1(T) > \len_1(R)$ and $\len_2(T) < \len_2(R)$, and we say that $T$ is a \emph{child} of $R$ if $T \succeq R$ and if $S \in \{ \tilde{R}_l\}$ and $T \succeq S \succeq R$ then $S = T$ or $S= R$.
We may define \emph{ancestors} and \emph{parents} similarly, with the relations reversed.

Fix a rectangle $R \in \{ \tilde{R}_l\}$, and consider $\prog \in R$ that lies in at least $n$ distinct rectangles of 
$\{ \tilde{R}_l\}$.
Then $\prog$ lies in at least $\lceil n/2\rceil$ distinct rectangles $S$ such that $S \succeq R$, or   
$\prog$ lies in at least $\lceil n/2\rceil$ distinct rectangles $S$ such that $S \preceq R$.
Thus 
\[
\begin{aligned}
&\labs\lset \prog \in R \colon \sum\nolimits_{l} \chi_{\tilde{R}_l}(\prog) \geq n \rset\rabs \\
&\qquad\leq 
\labs\lset \prog \in R \colon \sum\nolimits_{l}^{\succeq} \chi_{\tilde{R}_l}(\prog) \geq \frac{n}{2} \rset\rabs
+
\labs\lset \prog \in R \colon \sum\nolimits_{l}^{\preceq} \chi_{\tilde{R}_l}(\prog) \geq \frac{n}{2} \rset\rabs ,
\end{aligned}
\]
where $\sum^\succeq_{l}$ indicates that we sum over $l$ such that $\tilde{R}_l \succeq R$, and $\sum^\preceq_{l}$ is defined analogously.
We estimate the measure of one of these two sets: the other may be estimated similarly.

Consider all $S \in \{ \tilde{R}_l\}$ such that $S \succeq R$.
Inside this collection of rectangles, we may identify the children $S_1$, \dots, $S_p$ of $R$, which are pairwise disjoint, the children of the children of $R$, which are again pairwise disjoint, and so on.
Observe that 
\[
R \cap (S_1 \cup \dots \cup S_p) 
= \pi_1(R) \times (\pi_2(S_1) \cup \dots \cup \pi_2(S_p)) 
\subseteq \pi_1(R) \times \pi_2(R),
\]
and the near disjointness condition \eqref{eq:disjointness} implies that
\[
\labs \pi_2(S_1) \cup \dots \cup \pi_2(S_p) \rabs \leq \frac{1}{2} \labs \pi_2(R) \rabs.
\]
Then the measure of the set of all $\prog \in R$ that also belong to another rectangle $S \succ R$ is at most $ \labs R \rabs/2$.

Likewise, similar inequalities hold for the children of the children of $R$, which we may label as $T_1$, \dots, $T_q$, and so
\[
\labs \pi_2(T_1) \cup \dots \cup \pi_2(T_q) \rabs 
\leq \frac{1}{2} \labs \pi_2(S_1) \cup \dots \cup \pi_2(S_p) \rabs 
\leq \frac{1}{2^2} \labs \pi_2(R) \rabs.
\]
and the measure of the set of all $\prog \in R$ that also belong to two more rectangles $S, T \succ R$ is at most $2^{-2} \labs R \rabs$.
Continuing inductively, the measure of the set of all $\prog \in R$ that lie in at least $m$ distinct rectangles $S\succ R$ is at most $2^{-m} \labs R\rabs$.

Combining this estimate with an almost identical estimate for the measures of sets defined using ancestors and parents, we conclude that the measure of the set of all $\prog \in R$ that belong to at least $n$ rectangles of $\{ \tilde{R}_l \}$ is at most
\[
2 \times 2^{-\lceil (n-1)/2\rceil} \labs R\rabs \lesssim 2^{-n/2} \labs R\rabs.
\]
Then the measure of all $\prog \in G$ that belong to at least $n$ rectangles of $\{ \tilde{R}_l \}$ is at most a multiple of
\[
\sum_{l} 2^{-n/2} \biglabs \tilde{R}_l \bigrabs 
= 2^{-n/2} \sum_{l} \biglabs \tilde{R}_l \bigrabs
\lesssim 2^{-n/2} |E|,
\]
which is what we needed to prove.
\end{proof}

Now we can prove an endpoint estimate for the strong maximal function on $\proG$.
\begin{theorem}\label{thm Ms}
The strong maximal function satisfies the following:
\[ 
\big| \big\{\prog\in \proG\colon |\strmaxfn f(\prog)|>\lambda\big\}\big| 
\lesssim F_\Phi(f/\lambda)
\qquad\forall\lambda\in \Rplus.
\]
for all $f \in \LlogplusL(\proG)$.
\end{theorem}

\begin{proof}
We recall a fundamental estimate on the strong maximal function and pseudo\-dyadic strong maximal functions associated with adjacent pseudo\-dyadic systems, the proto\-type of which (in the setting of the product $\R\times \R$) was proved in \cite[Theorem 6.1]{LPW}.
More explicitly, we define the pseudo\-dyadic strong maximal function as follows:
\begin{align*}
\strmaxfn^{{\indext}_1,{\indext}_2}(f)(\prog)
:=\sup\lset \dashint_{R}\labs f(\proh) \rabs \wrt \proh\colon R \ni \prog, R \in \dcubes^{{\indext}_1} \times \dcubes^{{\indext}_2}  \rset,
\end{align*}
where ${\indext}_i=1,2,\dots ,\Indext_i$. 
Then
\begin{align}\label{eq:domination}
\strmaxfn(f)(\prog)\lesssim \sum_{{\indext}_1=1}^{\Indext_1}\sum_{{\indext}_2=1}^{\Indext_2} \strmaxfn^{{\indext}_1,{\indext}_2}(f)(\prog)
\qquad\forall \prog\in\proG.
\end{align}
This inequality means that instead of considering general products of balls, it suffices to consider rectangles, and Theorem \ref{thm Covering} may be applied.

In light of \eqref{eq:domination}, sublinearity and Proposition \ref{prop density}, it suffices to show that 
\begin{equation}\label{eq:enough-to-prove}
\big| \big\{g \in \proG\colon |\strmaxfn^{{\indext}_1,{\indext}_2} f(\prog)|>1\big\}\big| 
\lesssim F_\Phi(f)
\qquad\forall f \in \Leb2 \cap \LlogplusL (\proG),
\end{equation}
when ${\indext}_i=1,2,\dots ,\Indext_i$; we may assume that $f$ takes nonnegative real values.

Let $E=\big\{\prog \in \proG\colon |\strmaxfn^{{\indext}_1,{\indext}_2} f(\prog)|>1\big\}$.  
Since $f\in \Leb 2(\proG)$ and $\strmaxfn^{{\indext}_1,{\indext}_2}$ is $\Leb 2$ bounded, $|E|$ is finite.
For every $\prog\in E$ there exists $R_\prog\in \drect^{{\indext}_1,{\indext}_2}(\proG)$ satisfying
\begin{equation}\label{eq ms}
\dashint_{R_\prog} f(\proh) \wrt \proh > 1.  
\end{equation}
Then $E=\bigcup_{\prog\in E}  R_\prog$ by definition.

By Theorem \ref{thm Covering}, there is a sequence of rectangles $\{\tilde{R}_l\} \subseteq \{R_\prog\}_{\prog\in E}$ such that
\begin{equation}\label{eq cover}
|E| = \Bigl|\bigcup_{\prog\in E} R_\prog\Bigr| \lesssim \Bigl|\bigcup_{l} \tilde{R}_l\Bigr|
\qquad\text{and}\qquad
\Biglnorm\sum_{l} \chi_{\tilde{R}_l}\Bigrnorm_{\expe^{\Leb{}(\proG)}} 
\lesssim \Bigl|\bigcup_{\prog\in E} R_\prog\Bigr|.
\end{equation}

Write $\tilde{E}$ for $\bigcup_l \tilde{R}_l$ and $\tilde{E}_n$ for $\lset \prog \in \proG : \sum_{l} \chi_{\tilde{R}_l}(\prog) = n \rset$.
From \eqref{eq ms}, \eqref{eq:pre-Holder}, \eqref{eq:nearly-linear} and \eqref{eq:Jensen}, we deduce that $|E| \lesssim |\tilde{E}|$ and, for all $\lambda \in [1,\infty)$,
\begin{align*}
|\tilde{E}| 
&= \Biglabs \bigcup_{l} \tilde{R}_l \Bigrabs
 \le \sum_{l} |\tilde{R}_l| 
< \sum_{l} \int_{\tilde{R}_l} f(\prog)\wrt \prog\\
&= \int_{\tilde{E}} \sum_{l} \chi_{\tilde{R}_l}(\prog)f(\prog)\wrt \prog \\
&\leq \int_{\tilde{E}} \Phi(\lambda^2 f(\prog)) \wrt \prog 
+ \int_{\tilde{E}} \Psi( \sum_{l} \chi_{\tilde{R}_l}(\prog)/\lambda^2) \wrt \prog\\
&\leq \Phi(\lambda^2)\int_{\proG} \Phi(f(\prog)) \wrt \prog 
+ \frac{1}{\lambda} \int_{\tilde{E}} \Psi( \sum_{l} \chi_{\tilde{R}_l}(\prog)/\lambda) \wrt \prog.
\end{align*}
Since $|\tilde{E}_n| \leq C_1 2^{-n/2} |\tilde{E}|$, where $C_1$ is a geometric constant, 
\begin{align*}
|\tilde{E}| 
&\leq \Phi(\lambda^2) F_\Phi(f) 
+ \frac{1}{\lambda} \sum_{n=1}^{\infty} \expe^{n/\lambda}  \biglabs \tilde{E}_n\bigrabs
 \\
&\leq \Phi(\lambda^2) F_\Phi(f)  
+ \frac{C_1}{\lambda} \sum_{n=1}^{\infty} \expe^{n/\lambda} 2^{-n/2} \biglabs \tilde{E}\bigrabs .
\end{align*}
We fix $\lambda$ that is large enough that the series converges and the right hand term is less than $\biglabs \tilde{E} \bigrabs/2$; then $\biglabs \tilde{E} \bigrabs \leq 2 \Phi(\lambda^2) F_\Phi(f)$, as required.
\end{proof}

\begin{corollary}\label{cor:nontan-max-leq-strong-max}
The  maximal operator $\maxfn{\zeta}{\eta }$ satisfies the endpoint estimate:
\[ 
\bigl| \bigl\{\prog\in \proG\colon \labs \maxfn{\zeta}{\eta }(f)(\prog)\rabs > \lambda \bigr\}\bigr| 
\lesssim F_\Phi(f/\lambda)
\qquad\forall \lambda \in \Rplus 
\]
for all $f \in \LlogplusL(\proG)$.
\end{corollary}
\begin{proof}
It follows from \eqref{eq:Poisson-leq-strong} with a little care that $\maxfn{\zeta}{\eta }(f)\lesssim \strmaxfn f$ pointwise. 
\end{proof} 

The following lemma is another well known consequence of the boundedness of the strong maximal operator.
Recall that $\maxrect(U)$ denotes the collection of all maximal rectangles contained in $U$.

\begin{lemma}\label{lem:enlargement}
Let $U$ be an open subset of $\proG$ of finite measure and $\beta \in \Rplus$. 
Given $R$ in $\maxrect(U)$, let $\beta R$ be the set $\proz(R)\delta_\beta (\proz(R)^{-1}R)$.
Then 
\[
\bigglabs \bigcup_{R \in \maxrect(U)} \beta R \biggrabs \lesssim (\beta+1)^{\homodim_1 +\homodim_2} \labs U\rabs.
\]
\end{lemma}

\begin{proof}
Recall that $B(z(Q),c\len(Q) /3)\subseteq Q \subseteq B(z(Q),2C\len(Q))$ for all $Q \in \dcubes(G)$ and $P(\proz,\blen)$ is the product $B_1(z_1,\len_1) \times B_2(z_2,\len_2)$.
It follows that, if $R \in \drect$ and $\prog \in \beta R$, then 
\[
P(\proz(R), c \blen(R) /3) \subseteq R \subseteq P(\prog, 2C(\beta+1) \blen(R)), 
\]
and we deduce, much as in the proof of Lemma \ref{lem:compare-R-P}, that 
\[
\beta R \subseteq 
\bigglset \prog \in \proG \colon \strmaxfn (\chi_R)(\prog) \geq  C(\proG) (\beta+1)^{-\homodim_1-\homodim_2} \biggrset .
\]
Hence
\[
\bigglabs \bigcup_{R \in \maxrect(U)} \beta R \biggrabs
\leq \bigglabs \bigglset \prog \in \proG \colon \strmaxfn (\chi_R)(\prog) \geq C(\proG) (\beta+1)^{-\homodim_1-\homodim_2} \biggrset \biggrabs,
\]
and the hyperweak boundedness of $\strmaxfn$ implies that
\[
\bigglabs \bigcup_{R \in \maxrect(U)} \beta R  \biggrabs
\lesssim (\beta+1)^{\homodim_1 + \homodim_2}  F_\Phi(\chi_U)
\lesssim (\beta+1)^{\homodim_1+\homodim_2} \labs U\rabs,
\]
as claimed.
\end{proof}

\section{The atomic decomposition}

In this section, we first prove a hyperweak boundedness estimate for the Lusin area function $\sqfn{\fullnabla \pois}{1} (f)$, and then use this estimate to decompose $\LlogplusL(\proG)$ functions.

Recall that $\pti$ and $\qti$ denote the convolution kernels of the operators $e^{-t_i \sqrt{\Lap_i}}$ and $t_i \partial_{t_i} e^{-t_i \sqrt{\Lap_i}}$; then $\qti = t_i \partial_{t_i} \pti$.
We write $\ptt:=\ptone \otimes \pttwo$ and $\qtt:=\qtone \otimes \qttwo$.

\subsection{Hyperweak boundedness and the area function $\sqfn{\fullnabla \pois}{1} (f)$}

The first step is to apply Corollary \ref{cor:nontan-max-leq-strong-max}.

\begin{proposition}\label{thm MUs}
The nontangential maximal operator $\maxfn{\pois}{\eta }$ is hyperweakly bounded.
That is,
\[ 
\big| \big\{\prog \in \proG\colon |\maxfn{\pois}{\eta } f(\prog)|>\lambda \big\}\big| 
\lesssim F_\Phi(f/\lambda)
\qquad\forall \lambda \in \Rplus 
\]
for all $f \in \LlogplusL(\proG)$.
\end{proposition}

This weak type estimate for the Poisson maximal operator implies a similar estimate for the Lusin area function $\sqfn{\fullnabla \pois}{1}$, which is the key to establishing the atomic decomposition for $\LlogplusL (\proG)$ functions. 
More explicitly, we denote the tensor $\fullnabla\one\poisoneone(g_1) \otimes \fullnabla\two\poisonetwo(g_2)$ by $\fullnabla \pois(\prog)$ and define
\[ 
\sqfn{\fullnabla \pois}{\eta } f(\prog)
:=\bigglpar \int\!\!\!\dashint_{\Gamma(\prog)}
\left|  (f \ast (\fullnabla  \pois)_{\prot}) (\proh) \right| ^{2}
\wrt\proh  \wrtprotonprot \biggrpar^{1/2},
\]
where 
\begin{equation}\label{eq:def-intdashint}
\int\!\!\dashint_{\Gamma(\prog)} f(\proh,\prot) \wrt\proh \wrtprotonprot 
:= \int_{\proT}\dashint_{P(\prog,\prot)} f(\proh,\prot) \wrt\proh \wrtprotonprot.
\end{equation}

\begin{theorem}\label{thm Area}
The following estimate holds:
\begin{equation*}
\labs \lset \prog\in \proG\colon \sqfn{\fullnabla \pois}{1} (f)(\prog)>\lambda   \rset \rabs
\lesssim F_\Phi(f/\lambda)
\qquad\forall \lambda \in \Rplus 
\end{equation*}
for all $f \in \LlogplusL(\proG)$.
\end{theorem}

\begin{proof}
A recent result of Fan, Yan and the first and third authors \cite{CFLY} shows that for $f$ such that $\norm{ \maxfn{\pois}{\eta } f }_{\Leb1(\proG) }<\infty$,
\begin{align*}
\norm{  \sqfn{\fullnabla \pois}{1} f }_{\Leb1(\proG) }
\lesssim \norm{ \maxfn{\pois}{\eta } f }_{\Leb1(\proG) }.
\end{align*}
More precisely, define the sublevel set $L_{\eta }(\lambda) :=\lset \prog\in {\proG}\colon \maxfn{\pois}{\eta }(f)(\prog) \leq \lambda  \rset$.
Then it is shown that, when $\eta $ is sufficiently large,
\begin{align}\label{CFLY}
\biglabs \biglset \prog \in \proG\colon \sqfn{\fullnabla \pois}{1}(f) (\prog) >\lambda \bigrset \bigrabs
&\lesssim\biglabs L_{\eta }(\lambda)^c  \bigrabs
+
\frac{1}{\lambda^2} \int_{L_{\eta }(\lambda )} \maxfn{\pois}{\eta }^2 (f)(\prog) \wrt \prog
\end{align}
for all $\lambda\in \Rplus$.
Now $\biglabs L_{\eta }(\lambda)^c  \bigrabs \lesssim F_\Phi(f/\lambda)$ by Proposition \ref{thm MUs}, so the layer-cake formula and \eqref{eq:nearly-linear} show that
\begin{align*}
\biglabs \biglset \prog \in \proG\colon \sqfn{\fullnabla \pois}{1}(f) (\prog) >\lambda \bigrset \bigrabs
&\lesssim\biglabs L_{\eta }(\lambda)^c  \bigrabs
+
\frac{1}{\lambda^2} \int_{L_{\eta }(\lambda )} \maxfn{\pois}{\eta }^2 (f)(\prog) \wrt \prog \\
&= \biglabs L_{\eta }(\lambda)^c  \bigrabs
+
\frac{1}{\lambda^2} \int_0^\lambda 2\mu \Big|L_{\eta }(\mu)^c\Big|\wrt \mu\\
&\lesssim F_\Phi(f/\lambda)
+ 
\frac{1}{\lambda^2} \int_0^\lambda F_\Phi(f/\mu) \mu \wrt \mu\\
&\lesssim F_\Phi(f/\lambda)
+ 
\frac{1}{\lambda} \int_0^\lambda \log(\expe + \lambda/\mu) F_\Phi(f/\lambda) \wrt \mu\\
&\lesssim F_\Phi(f/\lambda) .
\end{align*}
This implies that the right-hand side of \eqref{CFLY} is finite. 
By repeating the argument of \cite{CFLY}, we see that \eqref{CFLY} also holds for $f\in \LlogplusL (\proG)$, and we conclude that
\begin{align*}
\big| \big\{\prog \in \proG\colon | \sqfn{\fullnabla \pois}{1} (f) (\prog)|>\lambda \big\}\big|
&\lesssim\big|  L_\eta (\lambda)^c\big|
+
\frac{1}{\lambda ^2}\int_{L_{\eta }(\lambda )} \maxfn{\pois}{\eta }^2 (f)(\prog)\wrt \prog\\
&\lesssim F_\Phi(f/\lambda)
\qquad\forall \lambda \in \Rplus
\end{align*}
for all $f \in \LlogplusL(\proG)$, as required. 
\end{proof}

Recall that $\qtt$ denotes $t_1\partial_{t_1} \pois_{t_1}\one \otimes t_2\partial_{t_2} \pois\two_{t_2}$ and
\[ 
\sqfn{\conjpois}{1}(f)(\prog)
:= \bigglpar \int\!\!\!\dashint_{\Gamma(\prog)}
\left|  f \ast \qtt (\proh) \right|^{2}
\wrt\proh \wrtprotonprot \biggrpar^{1/2}. 
\]

\begin{corollary}\label{cor:LlogL-to-weak-L1}
The area function $\sqfn{\conjpois}{1} (f)$ satisfies the estimate:
\begin{equation*}
\biglabs\biglset \prog\in \proG\colon \sqfn{\conjpois}{1} (f)(\prog)>\lambda   \bigrset \bigrabs
\lesssim F_\Phi(f/\lambda)
\qquad\forall \lambda \in \Rplus
\end{equation*}
for all $f \in \LlogplusL(\proG)$.
\end{corollary}

\begin{proof}
This holds as $f * \qtt$ is one of the components of the tensor  $f * (\slashed\nabla \pois)_{\prot} $.
\end{proof}

\subsection{The atomic decomposition for $\LlogplusL $}\label{Sec4.2}

Recall that the $\conjpois^{[i]}_1$ satisfy the standard smoothness, decay and cancellation conditions \eqref{eq:molecule} and \eqref{cancel psi}, and by \cite{GM} there exist compactly supported smooth functions $\phi^{[i]}$ on $G_i$ with integral $0$ such that
\begin{align*}
f
&= \int_\proT f \ast \qtt \ast \phitt  \wrtprotonprot  
\end{align*}
for all $f\in \Leb 2(\proG)$, where $\phitt=\phi_{t_1}\otimes\phi_{t_2}$.
This is the Calderón reproducing formula.
We suppose that $\supp(\phi^{[i]}) \subseteq  B_i(o,1)$, by rescaling if necessary.

\begin{theorem}\label{lemma L log L atom}
Let  $f\in \Leb 2(\proG) \cap \LlogplusL (\proG)$. 
Then we may write
\[
f=\sum_{k \in \Z} a_k,
\]
where the sum converges unconditionally in $L^2(\proG)$ and the \emph{atoms} $a_k$ satisfy: 
\begin{enumerate}
  \item $a_k$ vanishes outside a subset $U^{\dagger}_k$ of $\proG$ such that $|U^{\dagger}_k|\lesssim F_\Phi(2^{-k}{f})$;
  \item $\|a_k\|_{\Leb 2(\proG)}^2 \lesssim F_\Phi(2^k f)$;
  \item each $a_k$ can be further decomposed: $a_k=\sum_{R\in \maxrect(U^{*}_k)}a_{k,R} $, where  the sum converges unconditionally in $\Leb2(\proG)$, and
  \begin{enumerate}
    \item $\supp a_{k,R} \subseteq \beta R$, for a suitable $\beta \in \Rplus$;
    \item $\int_{G_1}a_{k,R}(g_1,g_2) \wrt g_1=0$ for all $g_2 \in G_2$;
    \item $\int_{G_2}a_{k,R}(g_1,g_2) \wrt g_2=0$ for all $g_1 \in G_1$;
    \item $\sum_{R\in \maxrect(U^*_k)}  \norm{a_{k,R}}^2_{\Leb 2(\proG)}\lesssim 2^{2k}F_\Phi(2^{-k} f)$.
  \end{enumerate}
\end{enumerate}
\end{theorem}

\begin{proof}
From Corollary \ref{cor:LlogL-to-weak-L1},  $\sqfn{\conjpois}{1}$ is hyperweakly bounded. 
Assume that $f\in \Leb 2(\proG) \cap \LlogplusL (\proG)$.
Then $\sqfn{\conjpois}{1} (f)\in \Leb {2}(\proG)$. 
Now for all $k\in\Z$, we set
\begin{align}
U_k
&:= \{ \prog\in\proG\colon \sqfn{\conjpois}{1}(f)(\prog)>2^k \}; 
\label{eq:def-Uk}\\
U^{*}_k
&:= \{\prog\in\proG\colon \strmaxfn (\chi_{U_k})(\prog)>  \alpha  \};
\label{eq:def-Ukstar}\\
U^{\dagger}_k
&:= \bigcup\nolimits_{R \in \maxrect(U^*_k)} \beta R;
\label{eq:def-Uktilde}\\
\bases_k
&:= \{ R \in \drect(\proG) \colon \labs R\cap U_k\rabs \geq \abs{R}/2, \labs R\cap U_{k+1}\rabs < \abs{R}/2 \};
\label{eq:def-Bk}
\end{align}
here $\alpha$ and $\beta$ are positive numbers that will be specified during the proof.
By definition $U_k \supseteq U_{k+1}$, so the sequence $\labs R\cap U_k\rabs / \abs{R}$ is decreasing and the sets $\bases_k$ are pairwise disjoint.
If $\prog \in R \in \bases_k$, then $|R \cap U_k | / \abs{R}> 1/2$ so $\prog \in U^{*}_k$ by Lemma \ref{lem:compare-R-P} provided that $\alpha < C(\proG)/2$; coupled with \eqref{eq:def-Bk}, this shows that
\begin{equation}\label{eq:key-inequality}
\labs R \cap (U^{*}_{k} \setminus U_{k+1}) \rabs 
= \labs R \setminus (R \cap U_{k+1}) \rabs 
\geq \frac{1}{2} \labs R \rabs.
\end{equation}

For each rectangle $R \in \bases_k$, we define the tent $R_+$ over $R$ to be the set
\begin{equation}\label{eq:def-tent}
\Biglset (\prog,\prot) \in \proG \times \proT\colon 
\prog\in R, 4C\len_1(R)  < t_1 \leq 8C\len_1(R), 4C\len_2(R) < t_2 \leq 8C\len_2(R)\Bigrset . 
\end{equation}
This definition implies that, if $R \in \bases_k$ and $(\proh,\prot) \in R_+$, then 
\begin{equation}\label{eq:R and R}
R 
\subseteq B_1(h_1,4C\len_1(R)) \times B_2(h_2,4C\len_2(R))  
\subseteq B_1(h_1,t_1) \times B_2(h_2,t_2)
= P(\proh, \prot)  .
\end{equation}
For future use, we note that, for any measurable set $V$ in $\proG$,
\begin{equation}\label{eq:conv-eqn}
\begin{aligned}
\chi_V * \chi_{P(\proo,\prot)}(\proh) 
&= \int_{\proG} \chi_V(\prog) \chi_{P(\proo,\prot)} (\prog^{-1}\proh) \wrt\prog 
= \int_{\proG} \chi_V(\prog) \chi_{P(\proh,\prot)} (\prog) \wrt\prog \\
&= \labs P(\proh, \prot) \cap V \rabs.
\end{aligned}
\end{equation}

Since $f\in \Leb 2(\proG)$, the reproducing formula implies that
\begin{align*}
f(\prog)
&= \int_\proT
f \ast \qtt \ast \phitt(\prog)  \wrtprotonprot  \\
&= \int_\proT \int_\proG
f \ast  \qtt(\proh) \phitt(\proh^{-1}\prog) \wrt \proh  \wrtprotonprot \\
&= \sum_{k \in \Z}\sum_{R\in \bases_k}  \iint_{R_+}
f \ast  \qtt(\proh) \phitt(\proh^{-1}\prog) \wrt \proh  \wrtprotonprot \\
&= \sum_{k \in \Z} a_k(\prog),
\end{align*}
say.
Next, for all $R\in \bases_k$, choose $\tilde{R}\in \maxrect(U^{*}_k)$ such that $R\subseteq \tilde{R}$. 
Then for all $S\in \maxrect(U^{*}_k)$, set $a_{k,S}=\sum_{R \in \bases_k\colon \tilde{R}=S } b_{k,R}$,
where
\[
b_{k,R}(\prog)
= \iint_{R_+} f \ast  \qtt(\proh) \phitt(\proh^{-1}\prog) \wrt \proh \wrtprotonprot .
\]
By construction,
\[
a_k= \sum_{S\in \maxrect(U^{*}_k)} a_{k,S}.
\]

Now 
\[
 \supp(\phitt) \subseteq B(o,t_1) \times B(o,t_2)\subseteq B_1(o,8C\len_1(R)) \times B_2(o,8C\len_2(R)) ,
\] 
and hence 
\begin{align*}
\supp( b_{k,R})  
&\subseteq B_1(z_1(R),10C\len_1(R)) \times B_2(z_2(R),10C\len_2(R))  \\
&= B_1(z_1(R),\beta c\len_1(R)/3) \times B_2(z_2(R),\beta c\len_2(R)/3) \\
&\subseteq \beta R,
  \end{align*}
where $\beta = 30 C/c$.
Thus $a_k$ vanishes on $(U^{\dagger}_k)^c$. 
By Lemmas \ref{lem:enlargement} and \ref{lem:compare-R-P} and Theorem \ref{lemma L log L atom}, 
\begin{equation}\label{eq:Uktilde-estimate}
|U^{\dagger}_k|\lesssim|U^{*}_k|\lesssim|U_k|\lesssim F_\Phi(2^{-k}f). 
\end{equation}

Next, we claim that for all $k\in \Z$,
\begin{align}\label{claim LlogplusL atom a k}
\|a_k\|_{\Leb 2(\proG)}^2\lesssim 2^{2k} F_\Phi(2^{-k} f).
\end{align}
To see this, take $h \in \Leb 2(\proG)$ such that $\|h\|_{\Leb 2(\proG)}=1$; then, writing $\check\phi$ for the reflected version of $\phi$, that is, $\check\phi(\prog) = \phi(\prog^{-1})$, we see that
\begin{align*}
\big|\langle a_k,h \rangle\big|
&= \bigg|\sum_{R\in \bases_k}
\iint_{R_+} f \ast \qtt(\prog) h \ast \check\phitt(\prog) \wrt \prog \wrtprotonprot\bigg|\\[5pt]
&\leq   
    \bigglpar \sum_{R\in \bases_k} \iint_{R_+} |f \ast \qtt(\prog)|^2 \wrt \prog\wrtprotonprot \biggrpar^{1/2} 
    \bigglpar \sum_{R\in \bases_k} \iint_{R_+} |h \ast \check\phitt(\prog)|^2 \wrt \prog \wrtprotonprot \biggrpar^{1/2} \\[5pt]
&\lesssim  
    \bigglpar \sum_{R\in \bases_k} \iint_{R_+} |f \ast \qtt(\prog)|^2 \wrt \prog \wrtprotonprot \biggrpar^{1/2}.
\end{align*}
By definition, \eqref{eq:key-inequality}, \eqref{eq:def-tent}, \eqref{eq:R and R}, a convolution identity and \eqref{eq:Uktilde-estimate},
\begin{align*}
&\sum_{R\in \bases_k} \iint_{R_+}  |f \ast \qtt(\proh)|^2 
\wrt \proh\wrtprotonprot \\
&\qquad\leq 2\sum_{R\in \bases_k} \iint_{R_+} \frac{\biglabs R \cap (U^{*}_{k} \setminus U_{k+1}) \bigrabs}{\abs{R}}
\labs f \ast \qtt(\proh)\rabs^2  \wrt \proh\wrtprotonprot \\
&\qquad\lesssim \sum_{R\in \bases_k} \iint_{R_+} 
\frac{\biglabs P(\proh,\prot) \cap (U^{*}_{k} \setminus U_{k+1}) \bigrabs}{|P(\proo,\prot)|} 
|f \ast \qtt(\proh)|^2 \wrt \proh\wrtprotonprot \\
&\qquad= \iint_{\proG \times \proT} 
\frac{\biglabs P(\proh,\prot) \cap (U^{*}_{k} \setminus U_{k+1}) \bigrabs}{|P(\proo,\prot)|}
|f \ast \qtt(\proh)|^2 \wrtprotonprot \wrt\proh\\
&\qquad= \iint_{\proG \times \proT} 
\chi_{U^{*}_{k} \setminus U_{k+1}} * \frac{\chi_{P(\proo,\prot)}}{|P(\proo,\prot)|}(\proh) 
|f \ast \qtt(\proh)|^2 \wrtprotonprot \wrt\proh \\
&\qquad= \int_{\proG}\chi_{U^{*}_{k} \setminus U_{k+1}} (\proh) \int_{\proT} 
\labs f \ast \qtt\rabs^2 *  \frac{\chi_{P(\proo,\prot)}}{|P(\proo,\prot)|}(\proh) 
\wrt \proh \wrtprotonprot \\\
&\qquad=\int_{U^{*}_k\backslash U_{k+1}}
|\sqfn{\conjpois}{1}(f)(\proh)|^2\wrt \proh \\
&\qquad\leq     2^{2k}|U^{*}_k|\\
&\qquad\lesssim 2^{2k}F_\Phi(2^{-k}f).
\end{align*}
This shows that $|\langle a_k,h\rangle|^2 \lesssim 2^{2k}F_\Phi(2^{-k}f)$, whence $\|a_k\|_{\Leb 2(\proG)}^2\lesssim 2^{2k}F_\Phi(2^{-k}f)$,
and \eqref{claim LlogplusL atom a k} holds.

We now claim that 
\begin{align}\label{claim LlogplusL atom a k 2}
\sum_{S\in \maxrect(U^*_k)}  \|a_{k,S}\|^2_{\Leb 2(\proG)}
&\lesssim 2^{2k}F_\Phi(2^{-k}f).
\end{align}
Arguing as above, we see that, if $h\in \Leb 2(\proG)$ and $\|h\|_{\Leb 2(\proG)}=1$, then
\begin{align*}
\big|\langle a_{k,S},h\rangle\big|
&= \bigg|
\sum_{R \in \bases_k(S)} \iint_{R_+}  f \ast  \qtt(\proh)\ h \ast  \check\phitt(\proh)  \wrt \proh \wrtprotonprot\bigg|\\[5pt]
&\leq  \bigglpar \sum_{R \in \bases_k(S)} \iint_{R_+}  |f \ast  \qtt(\proh)|^2  \wrt \proh \wrtprotonprot \biggrpar^{1/2} 
\\&\hspace{4cm}\times 
\bigglpar \sum_{R \in \bases_k(S)}\iint_{R_+} |h \ast \check\phitt(\proh)|^2  \wrt \proh \wrtprotonprot \biggrpar^{1/2}\\[5pt]
&\lesssim \bigglpar \sum_{R \in \bases_k(S)} \iint_{R_+} |f \ast  \qtt(\proh)|^2  \wrt \proh\wrtprotonprot \biggrpar^{1/2},
\end{align*}
where $\bases_k(S) := \{ R \in \bases_k\colon \tilde{R}=S \}$.
Hence
\[ 
\norm{a_{k,R}}_{\Leb 2(\proG)}\leq\bigglpar \sum_{R \in \bases_k(S)}\iint_{R_+} |f \ast  \qtt(\proh)|^2 \wrt\proh \wrtprotonprot \biggrpar^{1/2},
\]
which implies that 
\[ 
\sum_{R\in \maxrect(U^*_k)}  \norm{a_{k,R}}^2_{\Leb 2(\proG)}
\leq  \sum_{R\in \bases_k} \iint_{R_+} |f \ast  \qtt(\proh)|^2  \wrt \proh \wrtprotonprot.
\]
Then, much as argued to prove \eqref{claim LlogplusL atom a k}, we deduce \eqref{claim LlogplusL atom a k 2}.
The unconditional convergence of the sum follows as in \cite{CCLLO}.
\end{proof}

\section{Proof of Theorem \ref{thm main}}

In this section, we use the atomic decomposition to prove Theorem \ref{thm main}.
Recall that, if  $f\in \Leb 2(\proG) \cap \LlogplusL (\proG)$, then we may write
\[
f=\sum_{k \in \Z} a_k
\qquad\text{and}\qquad
a_k=\sum_{R\in \maxrect(U^*_k)}a_{k,R},
\]
where $U^*_k$ is as in Theorem \ref{lemma L log L atom}, and $a_k$ and $a_{k.R}$ satisfy support and cancellation conditions.
In Theorem \ref{lemma L log L atom}, we defined 
\[
U_k := \{ \prog\in\proG\colon \sqfn{\conjpois}{1}(f)(\prog)>2^k \}; 
\]
and
\begin{equation*}
U^{*}_k := \{\prog\in\proG\colon \strmaxfn (\chi_{U_k})(\prog)>\alpha \},
\end{equation*}
where $\alpha < C(\proG)/2$; now we also define 
\begin{equation}\label{eq:def-Ukstarstar}
U^{**}_k := \{\prog\in\proG\colon \strmaxfn (\chi_{U^{*}_k})(\prog)> \alpha \}.
\end{equation}
and
\begin{equation}\label{eq:def-Ukstarstarstar}
U^{***}_k := \{\prog\in\proG\colon \strmaxfn (\chi_{U^{**}_k})(\prog)> \alpha \}.
\end{equation}
\subsection{Proof of Theorem \ref{thm main} for area functions $\sqfn{\psi}{\eta }(f)$}\label{Sec4.3}

Recall from Section \ref{Sec4.2} that $\Leb 2(\proG) \cap \LlogplusL (\proG)$ is dense in $ \LlogplusL (\proG)$.
Now we take a general $\psi$ satisfying the decay, smoothness and cancellation conditions \eqref{eq:molecule} and \eqref{cancel psi}, and prove that
\begin{equation}\label{eq e40}
\big|\big\{ \prog\in \proG\colon \sqfn{\psi}{\eta }(f)(\prog)>\lambda   \big\}\big|
\lesssim F_\Phi(f/\lambda)
\qquad\forall \lambda \in \Rplus
\end{equation}
for all $f\in \LlogplusL (\proG)$.
By sublinearity and density, it suffices to prove that 
\begin{equation}\label{eq e40 homo}
\big|\big\{ \prog\in \proG\colon \sqfn{\psi}{\eta }(f)(\prog)>1  \big\}\big|
\lesssim F_\Phi(f)
\end{equation}
for all $f\in \Leb 2(\proG) \cap \LlogplusL (\proG)$.
We may and shall suppose that $\eta > 1$.

From Theorem \ref{lemma L log L atom}, we may write
\[ 
f=\sum_ka_k ,
\]
where the $a_k$ satisfy the conditions of the lemma.

Observe that
\begin{align*}
\Biglnorm \sum_{k \leq 0}a_k\Bigrnorm_{\Leb 2(\proG)}
&\le \sum_{k \leq 0} \norm{a_k}_{\Leb 2(\proG)} 
\lesssim  \sum_{k \leq 0} 2^{k}F_\Phi(2^{-k}f)^{1/2}
\lesssim F_\Phi(f)^{1/2}.
\end{align*}
Thus, by the $\Leb2(\proG)$ boundedness of $\sqfn{\psi}{\eta }$,
\begin{equation}\label{eq e41}
\bigg|\bigg\{\prog\in \proG\colon \sqfn{\psi}{\eta }\bigglpar\sum_{k \leq 0}a_k\biggrpar(\prog)>1\bigg\}\bigg|
\lesssim \Biglnorm \sum_{k \leq 0}a_k\Bigrnorm_{\Leb 2(\proG)}^2
\lesssim  F_\Phi(f).
\end{equation}

To handle $\sqfn{\psi}{\eta }(\sum_{k \geq 1} a_k)$, we use the fine structure of the atoms $a_k$: each $a_k$ is supported in $U_k^\dagger$, where $|U_k^\dagger|\lesssim 2^{-k}F_\Phi(f)$, and $a_k=\sum_{R\in \maxrect(U^*_k)} a_{k,R}$.

For all $R = {Q_1} \times {Q_2} \in \maxrect(U^*_k)$, let ${\tilde{Q}_1}$ be the biggest pseudo\-dyadic cube containing  ${Q_1}$ such that ${\tilde{Q}_1}\times {Q_2}\subset U^{**}_k$, where $U^{**}_k$ is defined in \eqref{eq:def-Ukstarstar}. 
Next, let ${\tilde{Q}_2}$ be the biggest pseudo\-dyadic cube containing ${Q_2}$ such that ${\tilde{Q}_1}\times {\tilde{Q}_2} \subseteq U^{***}_k$, where $U^{***}_k$ is defined in \eqref{eq:def-Ukstarstarstar}. 
Finally, let $R^{\dagger}$ be $100\beta_k({\tilde{Q}_1} \times {\tilde{Q}_2})$, where $\beta_k = 2^{k/ (2\homodim_1+2\homodim_2)}$.
By Lemmas \ref{lem:enlargement} and \ref{lem:compare-R-P} and Theorem \ref{lemma L log L atom},
\begin{equation}\label{eq:Rdagger-est}
\begin{aligned}
\bigglabs \bigcup_{R \in \maxrect(U^*_k)} R^{\dagger}\biggrabs
&\lesssim2^{k/2} |U_k^{\dagger}|
\lesssim2^{k/2} |U_k^{***}|
\lesssim2^{k/2} |U_k^{**}|  \\
&\lesssim2^{k/2} |U_k^*|
\lesssim2^{k/2} |U_k|
\lesssim 2^{k/2} F_\Phi(2^{-k} f).
\end{aligned}
\end{equation}

We claim that there exists $\delta \in \Rplus$ such that 
\begin{equation}\label{e 42}
\sum_{R\in \maxrect(U^*_k)}
\int_{(R^{\dagger})^c} |\sqfn{\psi}{\eta }(a_{k,R})(\prog)|\wrt\prog
\lesssim 2^{-\delta k} F_\Phi(f)
\qquad\forall k \in \N.
\end{equation}
Assume \eqref{e 42} for the moment, and let $E^\dagger :=\bigcup_{k \geq 1}\bigcup_{R \in \maxrect(U^*_k)} R^{\dagger}$.
Then on the one hand,
\begin{align*}
|E^\dagger| 
= \Big|\bigcup_{k \geq 1}\bigcup_{R \in \maxrect(U^*_k)}
R^{\dagger}\Big|
&\lesssim\sum_{k \geq 1} 2^{k/2} F_\Phi(2^{-k} f)
\lesssim F_\Phi(f)
\end{align*}
by \eqref{eq:Rdagger-est}.
On the other hand, summing \eqref{e 42} over positive $k$ shows that
\begin{align*}
\int_{(E^\dagger)^c}\Big|\sqfn{\psi}{\eta }\Big( \sum_{k \geq 1}a_k\Big)(\prog)\Big|\wrt \prog 
\lesssim F_\Phi(f).
\end{align*}
By Chebychev's inequality,
\begin{align*}
\Big|\Biglset \prog \in \proG \colon\sqfn{\psi}{\eta }\Big( \sum_{k \geq 1} a_k(\prog) \Big)>1
\Bigrset \Big|
&\leq |E^\dagger|+ \int_{(E^\dagger)^c}\Big|\sqfn{\psi}{\eta }\Biglpar 
\sum_{k \geq 1}a_k\Bigrpar (\prog)\Big|\wrt \prog
\lesssim F_\Phi(f).
\end{align*}
Together with \eqref{eq e41}, this implies \eqref{eq e40}. 

It remains to prove \eqref{e 42}.
Now
\begin{align*}
\int_{(R^{\dagger})^c} |\sqfn{\psi}{\eta }(a_{k,R})(g_1, g_2)|\wrt g_2 \wrt g_1
&\leq \int_{(100\beta_k{\tilde{Q}_1})^c}\int_{100{Q_2}}|\sqfn{\psi}{\eta }(a_{k,R})(g_1, g_2)|\wrt g_2 \wrt g_1 \\ 
&\qquad  +\int_{(100\beta_k{\tilde{Q}_1})^c}\int_{(100{Q_2})^c}|\sqfn{\psi}{\eta }(a_{k,R})(g_1, g_2)|\wrt g_2 \wrt g_1 \\
&\qquad  +\int_{(100\beta_k{\tilde{Q}_2})^c}\int_{100{Q_1}} |\sqfn{\psi}{\eta }(a_{k,R})(g_1, g_2)| \wrt g_1 \wrt g_2  \\
&\qquad  +\int_{(100\beta_k{\tilde{Q}_2})^c}\int_{(100{Q_1})^c} |\sqfn{\psi}{\eta }(a_{k,R})(g_1, g_2)| \wrt g_1 \wrt g_2 \\
&= \term{I}_1(R) + \term{I}_2(R) + \term{I}_3(R) + \term{I}_4(R),
\end{align*}
say, where $\beta_k=2^{k/(2\homodim_1+2\homodim_2)} \eta $.
It suffices to control $\term{I}_1(R)$ and $\term{I}_2(R)$, as the other two terms are similar.

By Hölder's inequality and the $ \Leb 2(G_2)$-boundedness of  $\sqfn{\psi^{[2]}}{\eta }$,
\begin{align*}
\term{I}_1
&\lesssim |{Q_2}|^{1/2} \int_{(100\beta_k {\tilde{Q}_1})^c} \bigglpar\int_{100{Q_2}} \biglpar \sqfn{\psi^{[1]}}{\eta }(a_{k,R})(g_1,g_2)\bigrpar ^2 \wrt g_2\biggrpar^{1/2} \wrt g_1.
\end{align*}

We use the cancellation and support restrictions on $a_{k,R}$ in the first variable, \eqref{eq:molecule}, homogeneity, the geometry ($h_1, z({Q_1}) \in {Q_1}$ and $g_1 \in (100\beta_k {\tilde{Q}_1})^c$) and Hölder's inequality to deduce that
\begin{align*}
&|a_{k,R} \ast_1 \psi_t^{[1]}(g_1,g_2)| \\
&\qquad=\bigglabs\int_{G_1} a_{k,R}(h_1,g_2)t^{-\homodim_1}
\bigl(\psi^{[1]}(\delta_{1/t}(h_1^{-1}g_1))-\psi^{[1]}(\delta_{1/t}(z({Q_1})^{-1}g_1))\bigr) \wrt{h}_1\biggrabs\\
&\qquad\lesssim \int_{G_1} \labs a_{k,R}(h_1,g_2)\rabs  
\frac{t^{-\homodim_1} \biglpar \rho(\delta_{1/t}(z({Q_1})^{-1}g_1)\delta_{1/t}(h_1^{-1}g_1)^{-1}) \bigrpar^{\epsilon}}
{\biglpar 1 + \rho(\delta_{1/t}(h_1^{-1}g_1)) + \rho(\delta_{1/t}(z({Q_1})^{-1}g_1))\bigrpar^{\homodim_1+2\epsilon}} \wrt{h}_1 \\
&\qquad= \int_{G_1} \labs a_{k,R}(h_1,g_2)\rabs 
\frac{t^{\epsilon} \biglpar \rho( z({Q_1})^{-1} h_1) \bigrpar^{\epsilon}}
{\biglpar t + \rho(h_1^{-1}g_1) + \rho(z({Q_1})^{-1}g_1)\bigrpar^{\homodim_1+2\epsilon}} \wrt{h}_1 \\
&\qquad\lesssim \int_{G_1} \labs a_{k,R}(h_1,g_2)\rabs 
\frac{t^{\epsilon} \len({Q_1})^{\epsilon}}
{\biglpar t + \rho(z({Q_1})^{-1}g_1)\bigrpar^{\homodim_1+2\epsilon}} \wrt{h}_1 \\
&\qquad\eqsim \frac{t^{\epsilon} \len({Q_1})^{\epsilon}}
{\biglpar t + \rho(z({Q_1})^{-1}g_1) \bigrpar^{\homodim_1+2\epsilon}} 
\norm{a_{k,R}(\cdot,g_2)}_{\Leb 1(G_1)},
\end{align*}
where $z({Q_1})$ is the center of ${Q_1}$.
Hence
\begin{align*}
\biglpar\sqfn{\psi^{[1]}}{\eta }(a_{k,R})(g_1,g_2) \bigrpar^2 
&=\int\!\!\!\dashint_{\Gamma_1^\eta (g_1)}
\biglabs a_{k,R} \ast_1 \psi_t^{[1]}(h_1,g_2) \bigrabs ^{2}
\wrt h_1  ,\frac{\wrt t}{t}\\
&\lesssim \int\!\!\!\dashint_{\Gamma_1^\eta (g_1)}
\frac{t^{2\epsilon} \len_1({Q_1})^{2\epsilon} \norm{a_{k,R}(\cdot,g_2)}_{\Leb 1(G_1)}^2}
{\biglpar t + \rho_1(z({Q_1})^{-1}h_1) \bigrpar^{2\homodim_1+4\epsilon}} 
\wrt h_1   \,\frac{\wrt t}{t}\\
&\eqsim \int\!\!\!\dashint_{\Gamma_1^\eta (g_1)}
\frac{t^{2\epsilon} \len_1({Q_1})^{2\epsilon} \norm{a_{k,R}(\cdot,g_2)}_{\Leb 1(G_1)}^2}
{\biglpar t + \rho_1(z({Q_1})^{-1}g_1) \bigrpar^{2\homodim_1+4\epsilon}} 
\wrt h_1  \,\frac{\wrt t}{t}\\
&= \frac{\len_1({Q_1})^{2\epsilon} \norm{a_{k,R}(\cdot,g_2)}_{\Leb 1(G_1)}^2 }
{ \biglpar \rho_1(z({Q_1})^{-1}g_1) \bigrpar^{2\homodim_1+2\epsilon} } 
 \int_{\Rplus}  \frac{t^{2\epsilon} }
{\biglpar t + 1 \bigrpar^{2\homodim_1+4\epsilon} } 
  \,\frac{\wrt t}{t}\\      
&\eqsim  \frac{\len({Q_1})^{2\epsilon}}{\biglpar\rho_1(z({Q_1})^{-1}g_1)\bigrpar^{2\homodim_1+2\epsilon}}  \norm{a_{k,R}(\cdot,g_2)}^2_{\Leb 1(G_1)}  \\
&\lesssim  \frac{\len({Q_1})^{2\epsilon} \labs {Q_1}\rabs}
{\biglpar\rho_1(z({Q_1})^{-1}g_1)\bigrpar^{2\homodim_1+2\epsilon}}   \norm{a_{k,R}(\cdot,g_2)}^2_{\Leb 2(G_1)} 
\end{align*}
(recall that $\dashint$ denotes the ``average integral'').
This estimate implies that
\begin{align*}
\term{I}_1(R)
&= \int_{(100\beta_k {\tilde{Q}_1})^c}\int_{100{Q_2}}\labs\sqfn{\psi}{\eta }(a_{k,R})(g_1,g_2)\rabs\wrt g_2 \wrt g_1  \\
&\leq \int_{(100\beta_k {\tilde{Q}_1})^c} |100{Q_2}|^{1/2} \lpar \int_{100{Q_2}}|\sqfn{\psi}{\eta }(a_{k,R})(g_1,g_2)|^2\wrt g_2
\rpar^{1/2} \wrt g_1\\
&\leq \int_{(100\beta_k {\tilde{Q}_1})^c} |100{Q_2}|^{1/2} \lpar \int_{G_2}|\sqfn{\psi}{\eta }(a_{k,R})(g_1,g_2)|^2\wrt g_2
\rpar^{1/2} \wrt g_1\\
&\lesssim \int_{(100\beta_k {\tilde{Q}_1})^c} |{Q_2}|^{1/2}
\frac{\len({Q_1})^{\epsilon} |{Q_1}|^{1/2} }{\rho_1(z({Q_1})^{-1}g_1)^{\homodim_1+\epsilon}} 
\norm{ a_{k,R}}_{\Leb2(\proG)} \wrt g_1\\
&=  \len({Q_1})^{\epsilon} 
\int_{(100\beta_k {\tilde{Q}_1})^c} \frac1{\rho_1(z({Q_1})^{-1}g_1)^{\homodim_1+\epsilon}} \wrt g_1 
\abs{R}^{1/2}  \norm{a_{k,R}}_{\Leb 2(\proG)} \\
&\eqsim  \beta_k ^{-\epsilon} \gamma_1(R)^{-\epsilon} \abs{R}^{1/2}  \norm{a_{k,R}}_{\Leb 2(\proG)}.
\end{align*}

To estimate $\term{I}_2(R)$, we argue as follows.
For $g_1\notin 100\beta_k {\tilde{Q}_1}$ and $g_2\notin 100 {Q_2}$, 
\begin{align*}
\biglpar\sqfn{\psi}{\eta }(a_{k,R})(g_1,g_2)\bigrpar^2
&=\int\!\!\!\dashint_{\Gamma^\eta (\prog)}  
\biglabs a_{k,R}\ast\psitt(\prog) \bigrabs ^2
\wrt\proh 
\wrtprotonprot ,
\end{align*}
and because of the cancellation of $a_{k,R}$ and properties of $\psi\one$ and $\psi\two$,
\begin{align*} 
&\biglabs a_{k,R}\ast\psitt(\prog) \bigrabs \\
&\qquad=  \bigglabs \int_{\proG} a_{k,R}(\proh) 
 \psitt(\proh^{-1}\prog) \wrt \proh \biggrabs \\
&\qquad=  \bigglabs \int_{\proG} a_{k,R}(h_1,h_2) 
 \psi\one_{t_1}(h_1^{-1}g_1)  \psi\two_{t_2}(h_2^{-1}g_2) \wrt h_2 \wrt h_1 \biggrabs \\ 
&\qquad=  \bigglabs \int_{\proG} a_{k,R}(h_1,h_2) 
\biglpar \psi\one_{t_1}(h_1^{-1}g_1) - \psi\one_{t_1}(z({Q_1})^{-1}g_1) \bigrpar\\
&\qquad\qquad\qquad\qquad \times\biglpar \psi\two_{t_2}(h_2^{-1}g_2) -\psi\two_{t_2}z({Q_2})^{-1}g_2)\bigrpar\wrt h_2 \wrt h_1 \biggrabs \\  
&\qquad\leq \norm{a_{k,R}}_{\Leb1(\proG)} 
\sup\lset \labs \psi\one_{t_1}(h_1^{-1}g_1) - \psi\one_{t_1}(z({Q_1})^{-1}g_1) \rabs \colon h_1 \in {Q_1}, g_1 \in (100b  {Q_1})^c \rset\\
&\qquad\qquad\qquad\qquad \times\sup\lset \labs \psi\two_{t_2}(h_2^{-1}g_2) - \psi\two_{t_2}(z({Q_2})^{-1}g_2) \rabs \colon h_2 \in {Q_2}, g_2 \in (100b  {Q_2})^c \rset\\
&\qquad\lesssim \norm{a_{k,R}}_{\Leb 1(\proG)} 
\frac{t_1^{\epsilon} \len({Q_1})^{\epsilon}}{\biglpar t_1 + \rho_1(g_1^{-1}z({Q_1}))\bigrpar^{\homodim_1+2\epsilon}}
\frac{t_2^{\epsilon}\len({Q_2})^{\epsilon}}{\biglpar t_2 + \rho_2(g_2^{-1}z({Q_2}))\bigrpar^{\homodim_2+2\epsilon}} \,,
\end{align*}
by the same geometrical arguments as used to treat $\term{I}_1(R)$, applied in both variables.
Hence $\biglpar\sqfn{\psi}{\eta }(a_{k,R})(\prog)\bigrpar^2$ is dominated by a multiple of $\norm{a_{k,R}}_{\Leb 1(\proG)} $ by
\begin{align*}
& \int\!\!\!\dashint_{\Gamma^\eta (\prog)}  
\frac{t_1^{2\epsilon} \len({Q_1})^{2\epsilon}}{\biglpar t_1 + \rho_1(h_1^{-1}z({Q_1}))\bigrpar^{2\homodim_1+2\epsilon}}
\frac{t_2^{2\epsilon} \len({Q_2})^{2\epsilon}}{\biglpar t_2 + \rho_2(h_2^{-1}z({Q_2}))\bigrpar^{2\homodim_2+2\epsilon}} 
\wrt\proh 
\wrtprotonprot \\
&\eqsim \int_{\proT}  
\frac{t_1^{2\epsilon} \len({Q_1})^{2\epsilon}}{\biglpar t_1 + \rho_1(g_1^{-1}z({Q_1}))\bigrpar^{2\homodim_1+4\epsilon}}
\frac{t_2^{2\epsilon} \len({Q_2})^{2\epsilon}}{\biglpar t_2 + \rho_2(g_2^{-1}z({Q_2}))\bigrpar^{2\homodim_2+4\epsilon}} 
\wrtprotonprot \\
&\eqsim   
\frac{ \len({Q_1})^{2\epsilon}}{\biglpar \rho_1(g_1^{-1}z({Q_1}))\bigrpar^{2\homodim_1+2\epsilon}}
\frac{ \len({Q_2})^{2\epsilon}}{\biglpar \rho_2(g_2^{-1}z({Q_2}))\bigrpar^{2\homodim_2+2\epsilon}} 
\,.
\end{align*}
Our estimation of $\term{I}_2(R)$ concludes with the observation that
\begin{align*}
\term{I}_2(R)
&= \int_{(100\beta_k {\tilde{Q}_1})^c}\int_{(100{Q_2})^c}| \sqfn{\psi}{\eta }(a_{k,R})(g_1,g_2)|\wrt g_2 \wrt g_1  \\
&\lesssim \norm{a_{k,R}}_{\Leb 1(\proG)} \int_{(100\beta_k {\tilde{Q}_1})^c}\int_{(100{Q_2})^c} 
\frac{ \len({Q_1})^{\epsilon}}{\biglpar \rho_1(g_1^{-1}z({Q_1}))\bigrpar^{\homodim_1+\epsilon}}
\frac{ \len({Q_2})^{\epsilon}}{\biglpar \rho_2(g_2^{-1}z({Q_2}))\bigrpar^{\homodim_2+\epsilon}}  
\wrt g_2 \wrt g_1 \\
&\lesssim \frac{ \len({Q_1})^{\epsilon}}{(100 \beta_k  \len({\tilde{Q}_1}))^{\epsilon}}
\frac{ \len({Q_2})^{\epsilon}}{(100  \len({Q_2}))^{\epsilon}} 
\abs{R}^{1/2} \norm{a_{k,R}}_{\Leb 2(\proG)}  \\
&\eqsim  \beta_k ^{-\epsilon} \gamma_1(R)^{-\epsilon} \abs{R}^{1/2} \norm{a_{k,R}}_{\Leb 2(\proG)}.
\end{align*}

Then, by combining the estimates of $\term{I}_1(R)$ and $\term{I}_2(R)$, and then using Hölder's inequality and Lemma \ref{theorem-cover lemma}, we see that for all $k \geq 1$,
\begin{align*}
&\sum_{R\in \maxrect(U^*_k)} \term{I}_1(R) + \term{I}_2(R)  \\
&\lesssim \beta_k ^{-\epsilon} \sum_{R\in  \maxrect(U^*_k)} \gamma(R)^{-\epsilon} \abs{R}^{1/2} \norm{a_{k,R}}_{\Leb 2(\proG)}\\
&\leq  \beta_k ^{-\epsilon}\bigglpar\sum_{R\in \maxrect(U^*_k)} \gamma_1(R)^{-2\epsilon}\abs{R}
\biggrpar^{1/2}\bigglpar\sum_{R\in \maxrect(U^*_k)}  \norm{a_{k,R}}^2_{\Leb 2(\proG)}\biggrpar^{1/2}\\
&\lesssim  \beta_k ^{-\epsilon} |U_k|^{1/2} \Biglpar 2^{2k} F_\Phi(2^{-k} f) \Bigrpar^{1/2}\\
&\lesssim  \beta_k ^{-\epsilon} 2^{k} F_\Phi(2^{-k} f)\\
&\lesssim  2^{-\epsilon k/(2\homodim_1+2\homodim_2)} F_\Phi(f) .
\end{align*}

Similar estimates hold for $\term{I}_3(R)$ and $\term{I}_4(R)$, but with $\gamma_2$ in place of $\gamma_1$, and hence \eqref{e 42} holds, and the proof is complete.

\subsection{Proof of Theorem \ref{thm main} for square functions $\sqfn{\psi}{0}(f)$}\label{Sec4.3b}

Arguing much as in Section \ref{Sec4.3}, we may also check that
\begin{equation}
\big|\big\{ \prog\in \proG\colon \sqfn{\psi}{\eta }(f)(\prog)>\lambda   \big\}\big|
\lesssim F_\Phi(f/\lambda)
\qquad\forall\lambda\in\Rplus
\end{equation}
for all $f\in \LlogplusL (\proG)$. 
We leave the details to the reader.

\subsection{Hyperweak boundedness of Riesz transformations}

We write $\Riesz$ for a double Riesz transformation $\Riesz\one_{j_1} \otimes \Riesz\two_{j_2}$.
\begin{theorem}\label{thm Riesz}
The double Riesz transformations satisfy the endpoint estimate:
\begin{equation}\label{eq Riesz}
\big|\big\{ \prog\in \proG\colon \Riesz(f)(\prog)>\lambda   \big\}
\lesssim  F_\Phi(f/\lambda)
\qquad\forall\lambda\in\Rplus
\end{equation}
for all $f\in \LlogplusL (\proG)$. 
\end{theorem}

\begin{proof}
By density and sublinearity, we need only prove that, for all $f\in \Leb 2(\proG) \cap \LlogplusL (\proG)$,
\begin{equation}\label{eq e52 homo}
\big|\big\{ \prog\in \proG\colon \Riesz(f)(\prog)>1  \big\}\big|
\lesssim  F_\Phi(f).
\end{equation}

From Theorem \ref{lemma L log L atom}, we may write
\[
f=\sum_ka_k
\]
where the $a_k$ satisfy the conditions specified there. 
By the $\Leb{2}(\proG)$ boundedness of $\Riesz$,
\begin{equation}\label{eq e53}
\Bigl|\Bigl\{\prog\in \proG\colon \Riesz\bigglpar\sum_{k \leq 0} a_k\biggrpar(\prog)>1 \Bigr\}\Bigr|
\lesssim \Bigl\|\sum_{k \leq 1}a_k\Bigr\|_{\Leb 2(\proG)}^2
\lesssim F_\Phi(f).
\end{equation}

To handle $\Riesz(\sum_{k \geq 1} a_k)$, we consider the $a_k$ in more detail.
The atom $a_k$ is supported in $U^\dagger_k$, where $|U^\dagger_k|\lesssim F_\Phi(2^{-k} f)$, and we may write $a_k=\sum_{R\in \maxrect(U^*_k)} a_{k,R}$.

Again, for all $R = {Q_1} \times {Q_2} \in \maxrect(U^*_k)$, let ${\tilde{Q}_1}$ be the biggest pseudo\-dyadic cube containing  ${Q_1}$ such that ${\tilde{Q}_1}\times
{Q_2}\subset U^{**}_k$, where $U^{**}_k$ is defined in \eqref{eq:def-Ukstarstar}. 
Next, let ${\tilde{Q}_2}$ be the biggest pseudo\-dyadic cube containing ${Q_2}$ such that ${\tilde{Q}_1}\times {\tilde{Q}_2} \subseteq U^{***}_k$, where $U^{***}_k$ is defined in \eqref{eq:def-Ukstarstarstar}. 
Now let $R^{\dagger}$ be $100\beta_k({\tilde{Q}_1} \times {\tilde{Q}_2})$, where $\beta_k = 2^{k/
(2\homodim_1+2\homodim_2)}$.
By Lemmas \ref{lem:enlargement} and \ref{lem:compare-R-P} and Theorem \ref{lemma L log L atom},
\begin{align*}
\bigglabs \bigcup_{R \in \maxrect(U^*_k)} R^{\dagger}\biggrabs
&\lesssim 2^{k/2} |U_k^{***}|
 \lesssim 2^{k/2} |U_k^{**}|
 \lesssim 2^{k/2} |U_k^*|
 \lesssim 2^{k/2} |U_k|
 \lesssim 2^{k/2} F_\Phi(2^{-k} f).
\end{align*}

As in the proof of Theorem \ref{thm main}, to prove that
\begin{equation}\label{e 54}
\bigg|\bigg\{\prog\in \proG\colon \Riesz\bigglpar\sum_{k \geq 1} a_k\biggrpar(\prog)>1\bigg\}\bigg|
\lesssim  F_\Phi(f),
\end{equation}
it suffices to show that, for some $\delta \in \Rplus$,
\begin{equation}\label{e 2}
\sum_{R\in \maxrect(U^*_k)}
\int_{(R^{\dagger})^c} |\Riesz(a_{k,R})(\prog)|\wrt \prog
\lesssim 2^{-\delta k} F_\Phi(f)
\end{equation}
for every $k \geq 1$.
If \eqref{e 2} holds, then \eqref{eq e53} and \eqref{e 54} imply \eqref{eq e52 homo}.
We now prove \eqref{e 2}.
\begin{align*}
\int_{(R^{\dagger})^c} |\Riesz(a_{k,R})(\prog)|\wrt \prog 
&\le \iint_{g_1\notin 100\beta_k{\tilde{Q}_1}} |\Riesz(a_{k,R})(\prog)|\wrt \prog +  \iint_{g_2\notin100\beta_k{\tilde{Q}_2}} |\Riesz(a_{k,R})(\prog)|\wrt \prog \\
&=:\term{I}_1 + \term{I}_2,
\end{align*}
say, where $\beta_k=2^{{k/(2\homodim_1+2\homodim_2)}}$.
Since the estimates of $\term{I}_1$ and $\term{I}_2$ are symmetric, we only estimate $\term{I}_1$.
Note that 
\begin{align*}
\term{I}_1&=\int_{g_1\notin 100\beta_k{\tilde{Q}_1}}\int_{100{Q_2}}|\Riesz(a_{k,R})(\prog)|\wrt \prog+\int_{g_1\notin 100\beta_k{\tilde{Q}_1}}\int_{(100{Q_2})^c}|\Riesz(a_{k,R})(\prog)|\wrt \prog \\
&=:\term{I}_{11}+\term{I}_{12},
\end{align*}
say.
By Hölder's inequality and the $\Leb 2(G_2)$-boundedness of  $\Riesz\two_{j_2}$,
\begin{align*}
\term{I}_{11}
&\lesssim |{\tilde{Q}_2}|^{1/2}\int_{g_1\notin 100\beta_k{\tilde{Q}_1}}
    \bigglpar\int_{Q_2} \biglpar \Riesz\one_{j_1}(a_{k,R})(g_1,g_2)\bigrpar ^2 \wrt g_2\biggrpar^{1/2} \wrt g_1.
\end{align*}
By the cancellation condition on $a_{k,R}(\cdot,g_2)$ and the smoothness of $\Riesz\one_{j_1}$,
\begin{align*}
\labs \Riesz\one_{j_1}(a_{k,R})(\prog) \rabs
&= \labs \int_{{G_1}} \biglpar \Riesz\one_{j_1}(g_1,h_1)-\Riesz\one_{j_1}(g_1,z({Q_1}))\bigrpar a_{k,R}(h_1,g_2) \wrt h_1 \rabs \\
&\lesssim \frac{\len({Q_1})^{\epsilon}}{\rho_1(g_1,z({Q_1}))^{\homodim_1+\epsilon}}
    \int_{{G_1}} a_{k,R}(h_1,g_2) \wrt h_1\\
&\lesssim \frac{\len({Q_1})^{\epsilon}}{\rho_1(g_1,z({Q_1}))^{\homodim_1+\epsilon}}
   |{Q_1}|^{1/2}\bigglpar\int_{{G_1}} |a_{k,R}(h_1,g_2)|^2 \wrt h_1\biggrpar^{1/2},
\end{align*}
where $z({Q_1})$ is the center of ${Q_1}$. 
Hence
\begin{align*}
\term{I}_{11}
\lesssim \beta_k^{-\epsilon}\gamma_1(R)^{-\epsilon} \abs{R}^{1/2} \norm{a_{k,R}}_{\Leb 2(\proG)}.
\end{align*}

To estimate $\term{I}_{12}$, we use the cancellation of $a_{k,R}$ to write $\Riesz(a_{k,R})(\prog)$ as
\begin{align*}
&\int_{\beta_k R} \biglpar  \Riesz\one_{j_1}(g_1,h_1) - \Riesz\one_{j_1}(g_1, z({Q_1}))\bigrpar 
                         \biglpar  \Riesz\two_{j_2}(g_2,h_2)-\Riesz\two_{j_2}(g_2, z({Q_2}))\bigrpar a_{k,R}(h_1,h_2) \wrt h_1 \wrt h_2,
\end{align*}
where $z({Q_2})$ is the center of ${Q_2}$. 
Using the smoothness of $\Riesz\one_{j_1}$ and $\Riesz\two_{j_2}$, we deduce that
\begin{align*}
\term{I}_{12}&\lesssim \int_{g_1\notin 100b{\tilde{Q}_1}}\int_{(100{Q_2})^c} \frac{\len({Q_1})^{\epsilon}}{\rho_1(g_1,z({Q_1}))^{\homodim_1+\epsilon}}
\frac{\len({Q_2})^{\epsilon}}{\rho_2(g_2,z({Q_2}))^{\homodim_2+\epsilon}}
\wrt g_1 \wrt g_2 \labs R\rabs^{1/2} \norm{a_{k,R}}_{\Leb 2(\proG)}\\
&\lesssim  b^{-\epsilon}\gamma_1(R)^{-\epsilon} \abs{R}^{1/2} \norm{a_{k,R}}_{\Leb 2(\proG)}.
\end{align*}
Then, by continuing as in Section \ref{Sec4.3}, we establish \eqref{e 2} and complete the proof.                    %
\end{proof}

\end{document}